\documentclass{siamart190516}

\usepackage{amscd,amsfonts,amsmath,amssymb}
\usepackage[toc,page]{appendix}
\usepackage{array}
\usepackage[english]{babel}
\usepackage{bigstrut} 
\usepackage{bbold}
\usepackage{calc}
\usepackage{caption}
\usepackage{color}
\usepackage{colortbl}
\usepackage{csquotes}
\usepackage{dsfont}
\usepackage{enumerate}
\usepackage{enumitem}
\usepackage{epsfig}
\usepackage{esint}
\usepackage{hyperref}
\usepackage{etoolbox}
\usepackage{float}
\usepackage{graphpap}
\usepackage{graphics}
\usepackage{graphicx}
\usepackage{latexsym}
\usepackage{lipsum}
\usepackage{lmodern}
\usepackage{marginnote}
\usepackage{mathrsfs}
\usepackage{mathtools}
\usepackage{mathabx}        
\usepackage{multirow}
\usepackage{pdfpages}
\usepackage{pgfplots}
\usepackage{pxfonts}
\usepackage{rotating}
\usepackage{subcaption}
\usepackage{tabularx,booktabs}
\newcolumntype{Y}{>{\centering\arraybackslash}X}
\usepackage{tikz}
\usetikzlibrary{shapes}
\usetikzlibrary{shapes.misc}

\usepgfplotslibrary{patchplots}
\usetikzlibrary{patterns, positioning, arrows}
\pgfplotsset{compat=1.13}
\tikzset{cross/.style={cross out, draw=black, fill=none, minimum size=2*(#1-\pgflinewidth), inner sep=0pt, outer sep=0pt}, cross/.default={2pt}}

\usepackage{times}
\usepackage{varwidth}
\usepackage{wrapfig}
\usepackage{xspace}

\usepackage[english]{babel}
\pgfplotsset{compat=newest}

\newtheorem{remark}{Remark}
\newcommand{\R}{\mathbb{R}}
\newcommand{\Yspace}{Y}
\renewcommand{\d}{\ \mathrm{d}}

\renewcommand{\d}{{\,\mathrm{d}}}
\newcommand{\norm}[1]{\left\lVert#1\right\rVert}
\newcommand{\restr}[1]{|_{#1}}
\newcommand{\lznormT}[1]{\left\lVert#1\right\rVert_{L^2(\Hcell)}}
\newcommand{\lznormO}[1]{\left\lVert#1\right\rVert_{L^2(\omega)}}

\newcommand{\hmesh}{\mathcal{M}_h}
\newcommand{\Hmesh}{\mathcal{T}_H}
\newcommand{\hcell}{\mathbf{c}}
\newcommand{\Hcell}{\mathbf{t}}
\newcommand{\HQuadW}{\mu}
\newcommand{\HQuadP}{q}
\newcommand{\HQuadPi}{\HQuadP_{\Hcell,i}}
\newcommand{\hQuadW}{\nu}
\newcommand{\hQuadP}{r}
\newcommand{\hQuadPi}{\hQuadP_{\hcell,i}}
\newcommand{\RightHandSide}{l}
\newcommand{\RightHandSideh}{l_h}
\makeatletter
\DeclareRobustCommand\onedot{\futurelet\@let@token\@onedot}
\def\@onedot{\ifx\@let@token.\else.\null\fi\xspace}
 
\def\ie{\emph{i.e}\onedot} 
\def\cf{\emph{cf}\onedot}

\def\etal{\emph{et al}\onedot}

\makeatother
\headers{two-scale FE approximation of a homogenized plate model}{M. Rumpf, S. Simon, and C. Smoch}

\title{
Two-Scale Finite Element Approximation of a Homogenized Plate Model
}

\author{
Martin Rumpf
\thanks{Institute for Numerical Simulation, University of Bonn, Endenicher Allee 60, 53115 Bonn, Germany
(\email{martin.rumpf@uni-bonn.de}, \email{s6stsimo@uni-bonn.de}, \email{s6chsmoc@uni-bonn.de} ).
}
\and Stefan Simon \footnotemark[1]
\and Christoph Smoch\footnotemark[1]
}
\begin{document}

\maketitle

\begin{abstract}
This paper studies the discretization of a homogenization and dimension reduction model
for the elastic deformation of microstructured thin plates proposed by Hornung, Neukamm, and Vel{\v{c}}i{\'c} \cite{HoNeVe14}.
Thereby, a nonlinear bending energy is based on a homogenized quadratic form 
which acts on the second fundamental form associated with the elastic deformation. 
Convergence is proved for a multi-affine finite element discretization 
of the involved three-dimensional microscopic cell problems 
and a discrete Kirchhoff triangle discretization of the two-dimensional isometry-constrained 
macroscopic problem. Finally, the convergence properties are numerically verified in selected test cases 
and qualitatively compared with deformation experiments for microstructured sheets of paper. 
\end{abstract}



%%%%%%%%%%%%%%%%%%%%%%%%%%%%%%%%%%%%%%%%%%%%%%%%%%%%%%%%%%%%%%
%%%%%%%%%%%%%%%%%%%%%%%%%%%%%%%%%%%%%%%%%%%%%%%%%%%%%%%%%%%%%%
%%%%%%%%%%%%%%%%%%%%%%%%%%%%%%%%%%%%%%%%%%%%%%%%%%%%%%%%%%%%%%
\section{Introduction}
In this paper, the numerical discretization and simulation of a homogenized model of a thin plate is investigated. 
The Kirchhoff model of thin plate elasticity was derived in the celebrated work by Friesecke~\etal  \cite{FrJaMue02b} as the 
$\Gamma$-limit of 3d-nonlinear elasticity in the limit for vanishing plate thickness.  A geometrically natural generalization of isometry-constrained bending functionals and their stationary solutions have been investigated by Hornung in \cite{Ho17}.

Hornung~\etal in \cite{HoNeVe14} combined this result with homogenization to derive a homogenized nonlinear plate model from 3d-nonlinear elasticity via simultaneous homogenization and dimension reduction. 
They considered a plate of vertical thickness $\delta$ with a three-dimensional periodic microstructure of size $\varepsilon$ composed by 
 elastic material which is homogenous in the vertical direction. The derived limit energy is the integral over a quadratic form acting on the second fundamental form of the elastic deformation and describing the effective deformation of the plate. This quadratic form describes the effective behavior in the simultaneous limit 
for vanishing plate thickness and size of the microstructure with fixed ratio $\gamma=\tfrac{\delta}{\varepsilon}$
and results from a minimization problem on the unit cell similar to the usual corrector problem in two-scale models
\cite{EEnHu03,AlBr05}.  
The case where $\varepsilon^2 \leq \delta \leq \varepsilon$ was discussed by Vel{\v{c}}i{\'c} in \cite{Ve15}.
Recently, B\"ohnlein~\etal extended this theory to homogenized prestrained plates in \cite{BoNePa23}. 
In \cite{NeOl15} Neukamm and Olbermann 
considered the spatial periodic homogenization in the context of nonlinear bending functionals for plates.
In this case, the limit functional is not simply a quadratic function of the second fundamental of the deformed plate.
In \cite{BeSc21} de Benito Delgado and Schmidt investigated the plate theory limit for heterogeneous multilayers with material parameters varying strongly in vertical direction.

In this paper we will pick up the model derived in \cite{HoNeVe14}, and derive a discretization and approximation of this model, using the heterogeneous multiscale method (HMM)  \cite{EMiZh05}. To this end, we consider the straightforward generalization to macroscopically varying microstructures. 

HMM was introduced in \cite{EEn03} to compute solutions to problems, where separation of (multiple) scales occurs. If the data of the problem is oscillating on a small scale (as in the problem described in this paper), it is often out of reach to fully resolve the data. Instead, one is interested in the effective macroscopic behavior of the problem. In most cases, apart from toy problems in 1d, the effective macroscopic problem cannot explicitly be written down, but has to be computed as the solution of a problem on the unit cell. In \cite{EMiZh05}, error estimates for HMM were given under the assumption that the solutions to the microscopic problems were given. In \cite{Ab05} and \cite{Ab06}, the fully discrete HMM for linear elliptic and elastic problems was analyzed.

Different approaches have been discussed for the numerical approximation of the nonlinear bending of thin plates. 
The minimizing of nonlinear bending energies leads to a fourth-order system of Euler-Lagrange equations.
The discrete Kirchhoff Triangle (DKT) is a non-conforming finite element, which only requires to fulfill 
the isometry constraint at nodes of the triangulation. A conforming finite element approach (cf.~\cite{Br07}) would require $C^1$-elements 
and a proper notion of discrete isometric deformations.
Bartels~\cite{Ba11} used DKT to approximate bending isometries in the case of deformations of thin elastic plates. 
He also took this model into account in \cite{Ba17} to approximate plate deformations in the F\"{o}ppl--von K\'{a}rm\'{a}n model in particular to verify a break of symmetry for deformations of smooth, circular cones. Furthermore, the DKT element was also used in \cite{BaBoNo17} to approximate bending isometries of bilayer plates. In \cite{RuSiSm22}, this work has been extended to the case of thin elastic shells described by parametrized surfaces. We pick up this approach as well for the macroscopic plate model.
As an alternative, Bonito \etal \cite{BoNoNt20} proposed a discontinuous Galerkin approach for isometric deformations of thin elastic plates.
In \cite{BoGuNo21} Bonito \etal analyzed the convergence of a local discontinuous Galerkin approach for prestrained plates. 

% Structure of the paper
This paper is organized as follows: In ~\cref{sec:HoNeVe14}, the homogenized plate bending model derived in \cite{HoNeVe14} is reviewed. In ~\cref{sec:HomBend}, the model we consider is formulated, the associated corrector problem is derived and  the isometry constraint is used to rewrite the bending energy such that it is quadratic in the Hessian of the deformation. Then, ~\cref{sec:ApproxHomTensor} presents the numerical approximation of the microscopic problem, and ~\cref{sec:DKTapprox} resumes the DKT-element and as the main result of this paper the convergence of the fully discrete two-scale problem to the continuous two-scale problem
is stated. The proof relies on $\Gamma$-convergence and uses the paradigm of the heterogeneous multiscale method. In ~\cref{sec:Implementation}, details on the implementation of the numerical method are presented. Finally, in ~\cref{sec:NumEx}, the results of several numerical experiments are shown and the 
convergence behavior is quantitatively analyzed. Furthermore, we compare the simulation results with mechanical experiments for microstructured sheets of paper.

%%%%%%%%%%%%%%%%%%%%%%%%%%%%%%%%%%%%%%%%%%%%%%%%%%%%%%%%%%%%%%
%%%%%%%%%%%%%%%%%%%%%%%%%%%%%%%%%%%%%%%%%%%%%%%%%%%%%%%%%%%%%%
%%%%%%%%%%%%%%%%%%%%%%%%%%%%%%%%%%%%%%%%%%%%%%%%%%%%%%%%%%%%%%
\section{
Homogenization and dimension reduction limit by Hornung \etal \cite{HoNeVe14} revisited}\label{sec:HoNeVe14}

Let $\Omega_\delta \coloneqq \omega \times \delta I$ be the reference configuration of a thin elastic plate, with a Lipschitz domain with piecewise $C^1$ boundary $\omega\subset \R^2$ as midsurface, $I = (-\frac12,\frac12)$ and thickness $\delta >0$. Let $W:\R^2\times \R^{3\times 3}\rightarrow \R$ be a nonlinear energy density which is periodic with period $1$ in the first component. In \cite{HoNeVe14}, a nonlinear elastic energy 
\begin{align}\label{eq:3dEnergy}
\frac{1}{\delta^3} \int_{\Omega_\delta} W(\frac{\xi'}{\varepsilon}, \nabla \phi_\delta(\xi)) \d \xi
\end{align}
is considered for the thin plate of thickness $\delta$ and composed of a microstructure of width $\varepsilon >0$,
which is homogeneous in vertical direction (cf.~\cref{fig:sketch} on the left).
Here, $\xi'=(\xi_1,\xi_2)$ represents the in plane coordinates. 
%
%The above model appears as the limit problem with a 
%simultaneous homogenization of a microstructured thin plate and a dimension reduction from 3D volume elasticity to 
%a 2D elastic bending model as it is studied by Hornung~\etal in the case where $Q^3$ is independent of the macroscopic variable $x$.
%B\"ohnlein~\etal \cite{BoNePa23} considered the generalized case of piecewise constant macroscopic dependence on grain domains. 
%Skipping the dependence on the macroscopic variable one takes into account deformations $\phi_\delta:\Omega_\delta \rightarrow \R^3$ of the reference configuration $\Omega_\delta \coloneqq \omega \times \delta I$ 
%of a thin plate with thickness $\delta >0$ and a microscopic scale parameter $\varepsilon >0$
%with an elastic energy 
%\begin{align}\label{eq:3dEnergy}
%\frac{1}{\delta^3} \int_{\Omega_\delta} W(\frac{\xi'}{\varepsilon}, \nabla \phi_\delta(\xi)) \d \xi
%\end{align}
%with $\xi'=(\xi_1,\xi_2)$, which is 
%properly rescaled in $\delta$ for an energy density $W$, which is homogeneous in vertical direction  
%fffffff
\begin{figure}[htbp]
\begin{tikzpicture}
\node (0,0) {\includegraphics[width= 0.98\linewidth]{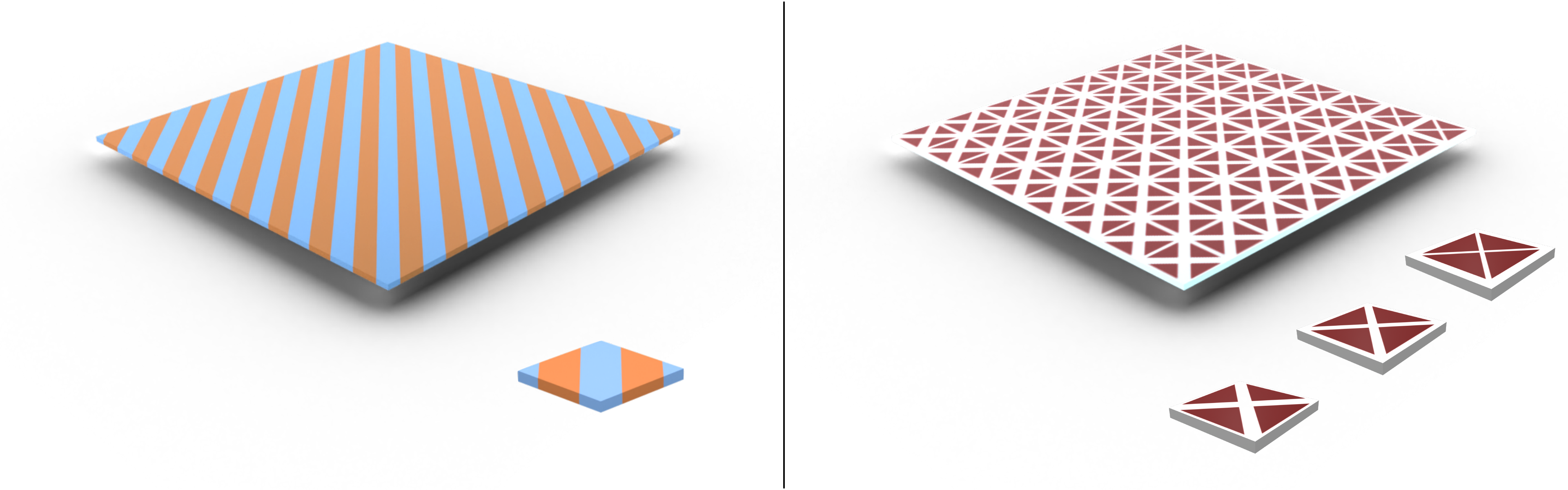}};
\node at (-2.55,-1.05) {$\delta$};
\node at (-2,-1.7) {$\varepsilon$};
\node at (-1.1,-1.7) {$\varepsilon$};
\draw[| - |] (-2.3,-1.02) -- (-2.3,-1.12);
\draw[| - |] (-2.2,-1.3) -- (-1.6,-1.55);
\draw[| - |] (-0.8,-1.3) -- (-1.4,-1.55);

\begin{scope}[xshift=5.3cm, yshift = -0.28cm]
\node at (-2.55,-1.05) {$\delta$};
\node at (-2,-1.7) {$\varepsilon$};
\node at (-1.0,-1.6) {$\varepsilon$};
\draw[| - |] (-2.3,-1.02) -- (-2.3,-1.12);
\draw[| - |] (-2.2,-1.25) -- (-1.5,-1.55);
\draw[| - |] (-0.9,-1.2) -- (-1.4,-1.55);
\end{scope}
\end{tikzpicture}
\caption{Thin plates with thickness $\delta$ and periodic in-plane microstructure with size $\varepsilon$,
left: homogeneous microstructure, right: Macroscopically varying microstructure.}
\label{fig:sketch}
\end{figure}
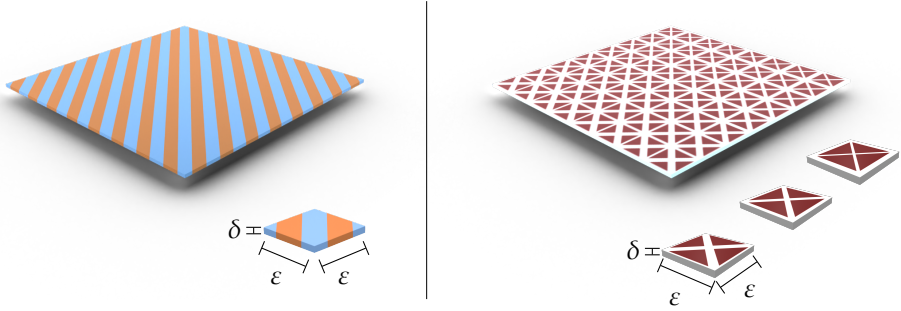
%fffffff
A rescaling of \eqref{eq:3dEnergy} onto the fixed 3D domain  $\Omega \coloneqq \omega \times I$
reads as
\begin{align}
\label{eq:rescaled3dEnergy}
\mathcal{W}^{\delta,\varepsilon}[\psi_\delta] \coloneqq \frac{1}{\delta^2} \int_\Omega W(\frac{z'}{\varepsilon}, \nabla_\delta \psi_\delta(z)) \d z\,,
\end{align}
for the rescaled deformation $\psi_\delta(z',z_3)\coloneqq \phi_\delta(z',\delta z_3)$ and the rescaled gradient 
$\nabla_\delta \psi_\delta \coloneqq (\nabla'\psi_\delta, \frac1\delta \partial_3 \psi_\delta)$ with the in-plane component 
$\nabla'\psi_\delta = (\partial_1\psi_\delta, \partial_2\psi_\delta)$. In \cite{HoNeVe14}, the authors studied the $\Gamma$-limit of \eqref{eq:rescaled3dEnergy} for $\delta$ and $\varepsilon$ converging to zero, under the following assumptions.
Given $\varepsilon= \varepsilon(\delta)$ as a monotone function from $(0,\infty)$ to $(0,\infty)$ one assumes that 
$\frac{\delta}{\varepsilon(\delta)} \rightarrow \gamma$ as $\delta \rightarrow 0$ with $\gamma \in (0,\infty)$.
Furthermore, the energy density 
$W: \R^2\times \R^{3\times3}\rightarrow [0,\infty]$ is Borel measurable and $W(\cdot,F)$ is periodic for all $F\in \R^{3\times3}$ with
$W(y',\cdot)\in C^0(\R^{3\times3},[0,\infty])$ for almost every $y'\in \R^2$. Furthemore, $W(y',RF) = W(y',F)$ for all $F \in \R^{3\times3}$ and all $R\in SO(3)$ and there exist positive constants $c_1,c_2,\rho$ such that
\begin{align*}
&W(y',F) \geq c_1 \text{dist}^2(F,SO(3)) \quad \text{for all }F\in \R^{3\times3},\\
&W(y',F) \leq c_2 \text{dist}^2(F,SO(3)) \quad \text{for all }F\in \R^{3\times3} \text{ with } \text{dist}^2(F,SO(3)) \leq \rho.
\end{align*}
Finally, a quadratic form $Q^3(y,\cdot)$ is defined as quadratic expansion of $W(y,\cdot)$ at the identity, \ie
there exists a monotone function $r:[0,\infty) \rightarrow [0, \infty]$ with $r(t) \rightarrow 0$ as $t\rightarrow0$, such that, for almost every $y'\in \R^2$ 
\begin{align}\label{eq:quadraticExpansion}
|W(y',I_3+G) - Q^3(y',G)| \leq r(|G|)|G|^2 \quad \text{for all }G \in \R^{3\times3}
\end{align}
with $\vert G \vert$ for $G\in \R^{3\times3}$ denoting the Frobenius norm.

\textit{Example. Consider a thin elastic plate as displayed on the left-hand side of ~\cref{fig:sketch}, consisting of two materials (blue and orange), oscillating on the scale $\varepsilon$. In the case of an isotropic material law, $Q^3$ is given by
\begin{align*}
Q^3(y', G) &= \mu(y') \left|\tfrac{G+ G^\top}{2}\right|^2 + \frac{\lambda(y')}{2} \left| \textrm{tr}\left(G\right) \right|^2
\end{align*}
where the Lam\'e-Navier parameters $\mu(y')$ and $\lambda(y')$ are periodic functions with period cell being the unit square $Y'\coloneqq (0,1)^2$
with different values on the orange and blue phase, respectively.  
%In the example displayed on the left-hand side of ~\cref{fig:sketch}, while the blue material is softer, with elastic modulus $\frac{1}{10}E$ and the same Poisson ratio. $\mu(y')$ and $\lambda(y')$ then vary on the unit square, forming diagonal stripes of orange and blue material: 
%\begin{align*}
%D &\coloneqq \{y'\in Y'\,:\; \frac14 < |y_1| + |y_2| \leq \frac34 \text{ or }   \frac54 < |y_1| + |y_2| \leq \frac74 \}\\
%\mu(y')&\coloneqq \begin{cases}
%\mu& \quad,\, y'\in D \\
%\frac{1}{10} \mu& \quad,\, y' \in Y'\setminus D
%\end{cases} \\
%\lambda(y')&\coloneqq \begin{cases}
%\lambda& \quad,\, y'\in D \\
%\frac{1}{10} \lambda& \quad,\, y' \in Y' \setminus D \,.
%\end{cases}
%\end{align*}.
}

Based on $Q^3$ now define 
\begin{align}
\label{eq:Q2gamma1}
Q^{2,\gamma}:\R^{2\times2}_\text{sym} \rightarrow \R\,;\quad Q^{2,\gamma}(A) \coloneqq \inf_{\vartheta \in \mathcal{V}} \int_{\Yspace} Q^3(y',\iota(y_3A) + \nabla_\gamma \vartheta(y))\d y
\end{align}
with $\Yspace\coloneqq Y' \times Y_3$ denoting the fundamental cell in 3D for $Y_3  \coloneqq (-\frac12,\frac12)$, 
$$\iota(A) \coloneqq  \begin{pmatrix} A & \begin{array}{r}0\\0\end{array}\\0\quad 0&0 \end{pmatrix}\in \R^{3\times 3}\,, $$
for $A \in \R^{2\times 2}$, 
and $\nabla_\gamma \vartheta(y) \coloneqq \left(\nabla_{y'} \vartheta(y), \frac1\gamma \partial_3 \vartheta(y)\right). $
Here, the infimum in \eqref{eq:Q2gamma1} is taken over all $\vartheta \in \mathcal{V}$, with
\begin{align*}
\mathcal{V}\coloneqq \big\{&\vartheta \in W^{1,2}(\Yspace;\R^3) \,:\, \vartheta(y) = (By',0)^\top + \varphi(y)\,,\\&B\in \R^{2\times2}_\text{sym}\,,\;\varphi \in W^{1,2}(\Yspace;\R^3) \,,\;\varphi(y) \text{ periodic in }y'\,,\;\int_{\Yspace} \varphi \d y= 0\,\big\}\,.
\end{align*}

Hornung \etal proved  in \cite{HoNeVe14} under the above assumptions the $\Gamma$-convergence of the energy $\mathcal{W}^{\delta,\varepsilon(\delta)}$ \eqref{eq:rescaled3dEnergy}, with respect to the strong $L^2$-topology, to the homogenized elastic plate energy 
\begin{align}
\label{eq:homBendEnergy1}
\mathcal{W}_\text{hom}^\gamma [\psi] = \begin{cases}
\int_\omega Q^{2,\gamma}(D^2\psi(x) \cdot n[\psi](x))\d x\,&,\, \psi \in \mathcal{A},\\
+\infty\,&\text{,\, else}
\end{cases}
\end{align}
with $\mathcal{A} \coloneqq \left\{ \psi \in W^{2,2}(\omega;\R^3)\colon \, \nabla \psi^\top \nabla \psi = I_2 \text{ a.e.} \right\}$  is the set of all $W^{2,2}$-isometric immersions of $\omega$ in $\R^3$ and $-D^2\psi \cdot n[\psi] = - \left(\partial_i\partial_j\psi \cdot n[\psi]\right)_{ij}\in \R^{2\times2}_\text{sym}$ 
is the second fundamental form corresponding to $\psi\in W^{2,2}(\omega;\R^3)$, 
with the unit normal $ n[\psi] =\frac{1}{|\partial_1 \psi \times \partial_2 \psi|} \partial_1 \psi \times \partial_2 \psi\in\R^3$.

Note that this is a generalization of \cite{FrJaMue02b}, where the plate model with homogeneous material was derived. 
In \cite{HoNeVe14}, the case $\gamma = \infty$ was also discussed, resulting in a different quadratic form $Q_\infty$. 
Here, we confine ourselves to the case $\gamma\in(0,\infty)$.

%%%%%%%%%%%%%%%%%%%%%%%%%%%%%%%%%%%%%%%%%%%%%%%%%%%%%%%%%%%%%%
%%%%%%%%%%%%%%%%%%%%%%%%%%%%%%%%%%%%%%%%%%%%%%%%%%%%%%%%%%%%%%
%%%%%%%%%%%%%%%%%%%%%%%%%%%%%%%%%%%%%%%%%%%%%%%%%%%%%%%%%%%%%%
\section{Homogenized plate bending energy model}\label{sec:HomBend}
In this paper, we investigate a numerical approximation scheme for a homogenized bending plate model. In \cite{HoNeVe14}, as revisited in ~\cref{sec:HoNeVe14}, solely a macroscopically homogeneous microstructure was taken into account. 
In \cite{BoNePa23}, B\"ohnlein \etal extended this theory to the case of piecewise constant macroscopic dependence on grain domains. 
Here, we consider the more general case of a macroscopically varying microstructure. 
Introducing the macroscopic variable $x\in\omega$ we consider the linearized elasticity density
\begin{align}
Q^3: \omega \times Y' \times \R^{3\times3} \rightarrow \R;\; (x,y',G) \to Q^3(x,y',G)
\end{align}
on the fundamental cell $\Yspace$, which is assumed to be quadratic in the third argument. 
In particular, for fixed macroscopic position $x$, the material properties are constant in the $y_3$ direction and depend solely on $y'$. 

In fact, we consider the homogenized elastic energy 
\begin{align}
\label{eq:homBendEnergy}
\mathcal{W}^\gamma [\psi] = \begin{cases}
\int_\omega Q^{2,\gamma}(x,D^2\psi(x) \cdot n[\psi](x))\d x\,&,\, \psi \in \mathcal{A},\\
+\infty\,&\text{,\, else,}
\end{cases}
\end{align}
where $x\in\omega$ denotes the macroscopic variable,
% and 
%$-D^2\psi \cdot n[\psi] = - \left(\partial_i\partial_j\psi \cdot n[\psi]\right)_{ij}\in \R^{2\times2}_\text{sym}$ 
%is the second fundamental form corresponding to $\psi\in W^{2,2}(\omega;\R^3)$, 
%with the unit normal $ n[\psi] =\frac{1}{|\partial_1 \psi \times \partial_2 \psi|} \partial_1 \psi \times \partial_2 \psi\in\R^3$
%
%Here, $\gamma \in (0,\infty)$ is a constant depending on the plate thickness and
%the microscopic cell size and will be specified later. 
%
where we assume that $Q^{2,\gamma}\colon\omega\times\R^{2\times2}_\text{sym}\rightarrow \R$ results from a 
microscopic optimization problem rescaled to the fundamental cell $\Yspace$ similar to \eqref{eq:Q2gamma1} and reads as
\begin{align}
\label{eq:Q2gamma}
Q^{2,\gamma}(x,A) = \inf_{\vartheta \in \mathcal{V}} \int_{\Yspace} Q^3(x,y',\iota(y_3A) + \nabla_\gamma \vartheta(y))\d y\,,
\end{align}

This is illustrated on the right-hand side of ~\cref{fig:sketch} for  
a material with macroscopically varying energy density reflecting two phases (red and white). Here, 
\begin{align*}
Q^3(x, y', G) &= \mu(x, y') \left|\tfrac{G+ G^\top}{2}\right|^2 + \frac{\lambda(x, y')}{2} \left| \textrm{tr}\left(G\right) \right|^2
\end{align*}
is defined on $\omega \times Y' \times R^{3\times 3}$ 
with Lam\'e-Navier parameters $\mu(x,y')$ and $\lambda(x, y')$ defined as functions on $\omega \times Y'$.
%The Lam\'e-Navier parameters $\mu(x,y')$ and $\lambda(x, y')$ are functions defined on $\omega \times Y'$. In the example displayed on the right-hand side of ~\cref{fig:sketch}, assume that the white material has elastic modulus $E$ and Poisson ratio $\nu$, while the red material is softer, with elastic modulus $\frac{1}{10}E$ and the same Poisson ratio. Let
%\begin{align*}
%D(a,b)\coloneqq \big\{ y'\in Y' \;:\; & 
%1- \tfrac{b}{\sqrt{2}} \vert y'- (1,0)^\top \vert_1  < 1+ \tfrac{b}{\sqrt{2}}  \vee
%1- \tfrac{b}{\sqrt{2}} \vert y'\vert_1 < 1+ \tfrac{b}{\sqrt{2}} \vee \\
%& y_1 < \tfrac{a}{2}  \vee  y_1 > 1- \tfrac{a}{2} \vee 
%y_2 < \tfrac{a}{2}  \vee  y_2 > 1- \tfrac{a}{2}
%\big\}\,.
%\end{align*}
%%with thicknesses $a(x_1)=(1-x_1) \frac{2-\sqrt{3}}{2}$, $b(x_1)=x_1 \frac{2-\sqrt{3}}{2}$
%$\mu(x, y')$ and $\lambda(x, y')$ are then defined as 
%\begin{align*}
%&\mu(x, y')\coloneqq \begin{cases}
%\mu& \quad,\,y'\in D((1-x_1) \frac{2-\sqrt{3}}{2}, x_1 \frac{2-\sqrt{3}}{2}) \\
%\frac{1}{10} \mu& \quad,\, \text{else}\,,
%\end{cases} \\
%&\lambda(x, y')\coloneqq \begin{cases}
%\lambda& \quad,\, y'\in D((1-x_1) \frac{2-\sqrt{3}}{2}, x_1 \frac{2-\sqrt{3}}{2}) \\
%\frac{1}{10} \lambda& \quad,\, \text{else} \,.
%\end{cases}
%\end{align*}

Finally, given a force $f \in L^2(\omega;\R^3)$ and clamped boundary conditions on a set $\Gamma_D \subset \partial\omega$, with $\mathcal{H}^1(\Gamma_D) > 0$ we ask for a minimizer of the homogenized and dimension reduced total free energy
\begin{align}\label{eq:fullEnergy}
{E} [\psi] = \begin{cases}
\mathcal{W}^\gamma[\psi] - \int_\omega f(x) \cdot \psi(x) \d x\,&,\, \psi \in \mathcal{A}^\text{BC}\\ +\infty\,&,\,\text{else}
\end{cases}
\end{align}
with constraint set 
$\mathcal{A}^\text{BC} \coloneqq \left\{ \psi \in \mathcal{A} :  \psi = \phi\,,\;\nabla\psi = \nabla\phi \text{ on }\Gamma_D \right\}$,
for some fixed $\phi \in \mathcal{A}$ representing the clamped boundary conditions.  In particular $\phi$ itself is in $\mathcal{A}^\text{BC}$.

%%%%%%%%%
%%%%%%%%%
%%%%%%%%%
%%%%%%%%%
%%%%%%%%%
%%%%%%%%%

\paragraph{A quadratic bending energy}
In what follows,  we rewrite the bending energy \eqref{eq:fullEnergy}.
As it was proved in \cite{BoNePa23}, by Poincare's and Korn's inequality using the $L^2$-orthogonality of $\iota(y_3A)$ and $ \nabla_\gamma \vartheta$ for $\vartheta \in \mathcal{V}$, for every $A \in \R^{2\times2}_\text{sym}$, and $x\in\omega$, there exists a unique $\vartheta(x,A,\cdot) \in \mathcal{V}$ such that
\begin{align}
\label{eq:Qgammavartheta}
Q^{2,\gamma}(x,A) = \int_{\Yspace} Q^3(x,y',\iota(y_3A) + \nabla_\gamma \vartheta(x,A,y))\d y \,.
\end{align}
This solution $\vartheta(x,A,\cdot)$ solves the Euler-Lagrange equation
\begin{align}
\label{eq:ELeq}
\int_{\Yspace} \left( \mathcal{C}^{3}(x,y')\nabla_\gamma \vartheta(y)\right): \nabla_\gamma v(y) \d y 
=- \int_{\Yspace}\left(\mathcal{C}^{3}(x,y')\iota(y_3A)\right): \nabla_\gamma v(y)\d y 
\end{align}
for all $v \in \mathcal{V}$, where $\mathcal{C}^3(x,y') \in \R^{3\times3\times3\times3}$ denotes 
the linearized elasticity tensor associated with the quadratic form $Q^3(x,y',\cdot)$
for $(x,y')\in \omega \times Y$, i.e.
\begin{align}
\label{eq:tensorDef}
\left( \mathcal{C}^{3}(x,y')(\iota(y_3A) + \nabla_\gamma \vartheta)\right):(\iota(y_3A) + \nabla_\gamma \vartheta) \coloneqq Q^3(x,y',\iota(y_3A) + \nabla_\gamma \vartheta)
\end{align}
for all $A\in\R^{2\times2}_\text{sym}$,  $\vartheta\in \mathcal{V}$. From its definition the symmetries 
$\mathcal{C}^3_{ijkl} = \mathcal{C}^3_{jikl} = \mathcal{C}^3_{ijlk}  = \mathcal{C}^3_{klij}$ for all $i,j,k,l\in\{1,2,3\}$
are deduced. 

Since $Q^3(x,y',\cdot)$ is quadratic, the solution $\vartheta(x,A,\cdot)$ of the Euler-Lagrange is linear in $A$ for every $x\in \omega$. Hence, $Q^{2,\gamma}(x,\cdot)$ is a quadratic form on $\R_\text{sym}^{2\times2}$ for every $x\in\omega$.
Furthermore, there are constants $\alpha,\; \beta \ge 0$ such that 
\begin{align} \label{eq:alphaBeta}
\frac{\alpha}{12} |A|^2 \leq Q^{2,\gamma}(x,A) \leq \frac{\beta}{12}|A|^2
\end{align}
for all $A\in \R_\text{sym}^{2\times2}$, see \cite{HoNeVe14}.

Using the isometry constraint $\nabla \psi^\top \nabla \psi = I_2$ one can simplify the nonlinear term 
$D^2\psi \cdot n[\psi]$ in \eqref{eq:fullEnergy}. 
A similar argument was presented in \cite{Ba15}, and adapted in \cite{RuSiSm22} for the case of isometric deformations of shells. In \cite[Proposition 1]{BoGuNo22}, an analogous result for a specific energy density in the prestrained case was derived using Christoffel symbol calculus.
\begin{proposition}
Let $\psi \in \mathcal{A}$. Then we have the identity
\begin{align*}
Q^{2,\gamma}(x, D^2\psi(x) \cdot n[\psi](x)) = \sum_{m=1}^3 Q^{2,\gamma}(x, D^2\psi_m(x))\,.
\end{align*}
\end{proposition}
\begin{proof}
In the following, we write $n= n[\psi]$ for the sake of simplicity.
Differentiation of $\partial_{i}\psi\cdot \partial_{i}\psi = 1$ for $i \in \{1,2\}$ in direction $j \in \{1,2\}$ yields
$\partial_j \partial_{i}\psi\cdot \partial_{i}\psi  =  0$.
Similarly, differentiation of $\partial_{i}\psi\cdot \partial_{j}\psi=0$ for $j \neq i$ in direction $i \in \{1,2\}$ gives
$ \partial_i^2 \psi \cdot \partial_{j}\psi + \partial_j \partial_{i}\psi\cdot \partial_{i}\psi = 0\,.$
Altogether, using that the parameter domain is two-dimensional, we obtain
$
\partial_i \partial_j \psi \cdot \partial_k \psi = 0 \ \ \forall i,j,k\in \{1,2\}.
$
Now, writing $\partial_i \partial_j \psi $ in terms of the orthonormal basis $\{ \partial_1 \psi, \partial_2\psi, n\}$, the identity
$
\partial_i \partial_j \psi = (\partial_i \partial_j \psi \cdot n)n+ (\partial_i \partial_j \psi \cdot \partial_1 \psi)\partial_1 \psi + (\partial_i \partial_j \psi \cdot \partial_2 \psi)\partial_2 \psi = (\partial_i \partial_j \psi \cdot n)n
$ holds.\\

For $x\in \omega$, since $Q^{2,\gamma}(x,\cdot)$ is quadratic, there exists a tensor $\mathcal{C}^{2,\gamma}(x) \in \R^{2\times2\times2\times2}$, such that for all $A\in\R^{2\times2}$
$$Q^{2,\gamma}(x,A) = \sum_{i,j,k,l=1}^2\mathcal{C}^{2,\gamma}_{ijkl}(x)A_{ij}A_{kl}\,.$$
Since $|n| = 1$ we can write
\begin{align*}
Q^{2,\gamma}(x,D^2\psi \cdot n) =& \sum_{i,j,k,l=1}^2\mathcal{C}^{2,\gamma}_{ijkl}(x)(\partial_i\partial_j \psi \cdot n)(\partial_k\partial_l \psi \cdot n)|n|^2 \\
=& \sum_{i,j,k,l=1}^2\mathcal{C}^{2,\gamma}_{ijkl}(x)\sum_{m=1}^3(\partial_i\partial_j \psi \cdot n)n_m(\partial_k\partial_l \psi \cdot n)n_m \\
=&\sum_{m=1}^3 \sum_{i,j,k,l=1}^2\mathcal{C}^{2,\gamma}_{ijkl}(x)\partial_i\partial_j \psi_m\partial_k\partial_l \psi_m = \sum_{m=1}^3 Q^{2,\gamma}(x, D^2\psi_m)\,.
\end{align*}
\end{proof}
In the following, with a slight misuse of notation, we write $Q^{2,\gamma}(x,D^2\psi)$ instead of\\ $\sum_{m=1}^3 Q^{2,\gamma}(x, D^2\psi_m)$. 
Hence the total free energy \eqref{eq:fullEnergy} reads as

\begin{align}
\label{eq:fullfullEnergy}
{E} [\psi] = \begin{cases}
\int_\omega Q^{2,\gamma}(x,D^2\psi(x) )\d x - \int_\omega f(x) \cdot \psi(x) \d x\,&,\, \psi \in \mathcal{A}^\text{BC},\\
+\infty\,&,\,\text{else.}
\end{cases}
\end{align}

Since $Q^{2,\gamma}(x,\cdot)$ is convex, by the lower order isometry constraint, 
the existence of minimizers for clamped boundary conditions induced by 
$\phi\in \mathcal{A}$ and force $f\in L^2(\omega;\R^3)$ follows by the direct method.

%%%%%%%%%%%%%%%%%%%%%%%%%%%%%%%%%%%%%%%%%%%%%%%%%%%%%%%%%%%%%%
%%%%%%%%%%%%%%%%%%%%%%%%%%%%%%%%%%%%%%%%%%%%%%%%%%%%%%%%%%%%%%
%%%%%%%%%%%%%%%%%%%%%%%%%%%%%%%%%%%%%%%%%%%%%%%%%%%%%%%%%%%%%%
\section{Discretization of the microscopic problem}\label{sec:ApproxHomTensor}
In this section, we investigate the finite element approximation of the microscopic minimization problem \eqref{eq:Q2gamma}. 
In what follows, we will use generic constants in the estimates.  
Let $\hmesh$ be a regular hexahedral mesh partition of $\Yspace$, with cells $\hcell \in\hmesh$ and define the finite element space
\begin{align*}
\mathcal{V}_h \coloneqq \big\{ &\vartheta_h \in C^0(\Yspace;\R^3)\,:\; 
\vartheta_h(y ) = (By',0)^\top + \varphi_h(y)\,,\\&B\in \R^{2\times2}_\text{sym}\,,\;\varphi_h|_\hcell \text{ multi-affine } 
\forall \hcell\in\hmesh \,,\;\varphi_h(y) \text{ $Y$-periodic in }y'\,,\int_{\Yspace} \varphi_h = 0\d y\big\}.
\end{align*}
On every $\hcell\in\hmesh$ we consider a numerical quadrature scheme with quadrature points 
$\hQuadPi= (\hQuadPi',\hQuadP_{\hcell,3,i})\in \hcell, \,i=1\dots k$, and weights $\hQuadW_i >0$, which is exact on tri-affine functions.
Furthermore, we assume that the set of quadrature points $\{\hQuadPi\}_{i=1,\ldots,k}$ is unisolvent 
with respect to the set of tri-affine functions in the sense of \cite[Section 2.3]{Ci78}. 
Actually, we consider a quadrature scheme defined on a reference cell and transferred by the tri-affine reference map to the actual cell. Then, for $A\in \R^{2\times2}_\text{sym}$, and for a macroscopic position $x\in\omega$ we make use of this quadrature and define the discrete quadratic form 
\begin{align}
\label{eq:Q_hinf}
Q_h^{2,\gamma}(x,A) \coloneqq 
\inf_{\theta_h \in \mathcal{V}_h}\Big( \sum_{\hcell\in\hmesh} |\hcell|
\sum_{i=1}^k \hQuadW_iQ^3(x,\hQuadPi',\iota(\hQuadP_{\hcell,3,i} A) + \nabla_\gamma \theta_h(\hQuadPi)) \Big)
\end{align}
as the discrete counterpart of \eqref{eq:Q2gamma}.
 Let us remark that microscopic approximation error $Q_h^{2,\gamma}(x,A)- Q^{2,\gamma}(x,A)$ is numerically evaluated and referenced in the numerical estimates 
\cref{eq:TotalErr} solely at macroscopic quadrature points (cf.~\cref{sec:DKTapprox}).

The above assumptions on the exactness of the quadrature scheme and the unisolvent property of the set of quadrature nodes implies that
\begin{align} \label{eq:alphaBetah}
\frac{\alpha}{12} |A|^2 \leq Q^{2,\gamma}_h(x,A) \leq \frac{\beta}{12}|A|^2
\end{align}
for all $A\in \R_\text{sym}^{2\times2}$, using the strong coercivity of  $Q^3(x,y',\cdot)$.
%\ie. 
%\begin{align}
%\label{eq:microquadratureRule}
%\frac{1}{|\hcell|} \int_\hcell \varphi(y) \d y = \sum_{i=1}^k \hQuadW_i \varphi(\hQuadPi) \qquad \text{for all }\varphi \text{ multi-affine}\,.
%\end{align}

%In what follows, let us assume that the microscopic quadrature rule is exact on constant functions. 
This ensures that there exists a unique $\vartheta_h(x,A,\cdot)\in \mathcal{V}_h$ with
\begin{align}
\label{eq:Qgamma_hvartheta_h}
Q^{2,\gamma}_h(x,A)=\sum_{\hcell\in\hmesh} |\hcell|\sum_{i=1}^k \hQuadW_iQ^3(x,\hQuadPi',\iota(\hQuadP_{\hcell,3,i} A) 
+ \nabla_\gamma \vartheta_h(x,\hQuadPi,A)) \,,\end{align}
which solves the linear system
\begin{align}
&\sum_{\hcell\in\hmesh} |\hcell|\sum_{i = 1}^k \hQuadW_i
\left( \mathcal{C}^3(x,\hQuadPi') \nabla_\gamma \vartheta_h(x,\hQuadPi,A)\right):  \nabla_\gamma v_h(\hQuadPi) \nonumber \\
&=-  \sum_{\hcell\in\hmesh} |\hcell|\sum_{i=1}^k\hQuadW_i \left( \mathcal{C}^3(x,\hQuadPi')\iota(\hQuadP_{\hcell,3,i}A)\right):  \nabla_\gamma v_h(\hQuadPi) \label{eq:ELeq_h}
\end{align}
with $\hQuadPi = (\hQuadPi',\hQuadP_{\hcell,3,i})$ and for all $v_h\in\mathcal{V}_h$.

The following proposition gives estimates for the finite element approximation of the microscopic problem:
\begin{proposition}
\label{prop:EHMM}
Let $A \in \R^{2\times2}_\text{sym}$ and $x\in\omega$ and suppose that the elasticity tensor $\mathcal{C}^3$ is bounded in  
$W^{1,\infty}(\omega \times Y', \R^{3\times3\times3\times3}))$. Furthermore, we assume that the quadrature scheme is exact on the set tri-affine functions and unisolvent with respect to this set.
Then for $\vartheta(x,A,\cdot) \in \mathcal{V}$ and $\vartheta_h(x,A,\cdot)\in\mathcal{V}_h$ denoting the solutions to \eqref{eq:ELeq} and \eqref{eq:ELeq_h}, respectively, 
there exists a constant $C>0$ depending on $\gamma$ and on $\mathcal{C}^3$, such that
\begin{align}\label{eq:ErrorHMM}
\norm{\vartheta_h(x,A,\cdot) - \vartheta(x,A,\cdot)}_{W^{1,2}(\Yspace;\R^3)} \!\leq\!  C  h |A|\,, \;
\left| Q_h^{2,\gamma}(x,A) - Q^{2,\gamma}(x,A)  \right| \!\leq\!  C h |A|^2\,.
\end{align} 
If the quadrature in the definition of $Q_h^{2,\gamma}(x,A)$ (cf.~\eqref{eq:Q_hinf}) is exact, i.e.
\begin{align}
\label{eq:ConsistencyCondition}
\int_\hcell \left(\mathcal{C}^3(x,y') \nabla u_h(y) \right):\nabla v_h(y) \d y=  |\hcell|\sum_{i=1}^k\hQuadW_i\left( \mathcal{C}^3(x,\hQuadPi')\nabla u_h(\hQuadPi)\right): \nabla v_h(\hQuadPi)\,,
\end{align}
for all  $u_h,v_h \in \mathcal{V}_h$ and all $x\in \omega$, then 
$\left| Q_h^{2,\gamma}(x,A) - Q^{2,\gamma}(x,A)  \right| \leq  C h^2 |A|^2\,$.
\end{proposition}
\begin{proof}
Let $A \in \R^{2\times2}_\text{sym}$ and $x\in\omega$ be arbitrary.
We define the bilinear forms for the left-hand side of \eqref{eq:ELeq}, \eqref{eq:ELeq_h}, respectively, 
\begin{align*}
a(x,\vartheta,v)& \coloneqq \int_{\Yspace}\left( \mathcal{C}^3(x,y')\nabla_\gamma \vartheta(y)\right): \nabla_\gamma v(y) \d y\,,\\
a_h(x,\vartheta_h,v_h)& \coloneqq \sum_{\hcell\in\hmesh} |\hcell|\sum_{i=1}^k\hQuadW_i
\left( \mathcal{C}^3(x,\hQuadPi')\nabla_\gamma \vartheta_h(\hQuadPi)\right): \nabla_\gamma v_h(\hQuadPi) \,,
\end{align*}
and, the linear forms  for the right-hand side of \eqref{eq:ELeq}, \eqref{eq:ELeq_h}, respectively,
\begin{align*}
\RightHandSide(x,v) &\coloneqq \int_{\Yspace} \left( \mathcal{C}^3(x,y')\iota(y_3A)\right): \nabla_\gamma v(y)\d y \,,\\
\RightHandSideh(x,v_h) &\coloneqq  \sum_{\hcell\in\hmesh} |\hcell|\sum_{i=1}^k\hQuadW_i\left( \mathcal{C}^3(x,\hQuadPi')\iota(\hQuadP_{\hcell,3,i}A)\right): \nabla_\gamma v_h(\hQuadPi) \,.
\end{align*}
By Strang's first lemma \cite[Theorem 4.1.1]{Ci78}, using the ellipticity property of $Q^3(x, \cdot)$, there exists a constant $C >0 $, such that
\begin{align*}
&\norm{\vartheta_h(x,A,\cdot) - \vartheta(x,A,\cdot)}_{W^{1,2}(\Yspace;\R^3)}\\ &\leq C\inf_{v_h \in \mathcal{V}_h}\left\{ \norm{v_h - \vartheta(x,A,\cdot)}_{W^{1,2}(\Yspace;\R^3)} + \sup_{u_h \in \mathcal{V}_h} \frac{\left| a(x,v_h,u_h) - a_h(x,v_h,u_h) \right|}{\norm{u_h}_{W^{1,2}(\Yspace;\R^3)}} \right\}
\\ & \quad + C \sup_{u_h \in \mathcal{V}_h}\frac{\left| \RightHandSide(x,u_h) - \RightHandSideh(x,u_h) \right|}{\norm{u_h}_{W^{1,2}(\Yspace;\R^3)}}\,.
\end{align*}
Classical elliptic regularity theory \cite[Theorem 8.12]{GiTr83}, adapted to linearized elasticity on the domain $Y=Y'\times Y_3$ with periodic boundary conditions on $\partial Y' \times Y_3$
and Neumann boundary conditions on $Y' \times \{-\tfrac12,\tfrac12\}$ implies 
$\vartheta(x,A,\cdot) \in W^{2,2}(\Yspace;\R^3)$. 
Furthermore, by the linearity of $\RightHandSide$ one gets
\begin{align} \label{eq:LinearOnA}
\norm{\vartheta}_{W^{2,2}(\Yspace;\R^3)} \leq C |A|\,.
\end{align}
By the Sobolev embedding theorem the Lagrangian interpolation $\mathcal{I}_h:W^{2,2}(\Yspace;\R^3)\rightarrow \mathcal{V}_h$ 
based on nodal interpolation is well-defined. Making use of the regularity of $\hmesh$  the Bramble-Hilbert-Lemma \cite{Br07}, and \eqref{eq:LinearOnA} imply
\begin{align}
\label{eq:ApproximationError}
\norm{\mathcal{I}_h[\vartheta(x,A,\cdot)] - \vartheta(x,A,\cdot)}_{W^{1,2}(\Yspace)} \leq Ch\norm{ \vartheta(x,A,\cdot)}_{W^{2,2}(\Yspace;\R^3)}  \leq C h|A|\,.
\end{align}
Furthermore, applying the consistency error estimate in \cite[Theorem 4.1.4 and Theorem 4.1.5´]{Ci78} involving the assumption on the numerical quadrature,
and \eqref{eq:LinearOnA} one obtains
\begin{align}
\label{eq:ConsistencyError1}
&\sup_{u_h \in \mathcal{V}_h} \frac{\left| (a-a_h)(x,\mathcal{I}_h[\vartheta(x,A,\cdot)],u_h) \right|}{\norm{u_h}_{W^{1,2}(\Yspace;\R^3)}}\leq  C  h \norm{ \vartheta(x,A,\cdot)}_{W^{2,2}(\Yspace;\R^3)} \leq C h \vert A \vert\,,\\
\label{eq:ConsistencyError2}
&\sup_{u_h \in \mathcal{V}_h}\frac{\left| (\RightHandSide-\RightHandSideh)(x,u_h)  \right|}{\norm{u_h}_{W^{1,2}(\Yspace;\R^3)}} \leq C h|A| \,,
\end{align}

Using the regularity \eqref{eq:LinearOnA} of the microscopic solution combined with the approximation error \eqref{eq:ApproximationError} and the consistency errors \eqref{eq:ConsistencyError1}, \eqref{eq:ConsistencyError2}, the first Strang Lemma ensures the first estimate in \eqref{eq:ErrorHMM}.

With respect to the second estimate in \eqref{eq:ErrorHMM}, we get with the notation $\vartheta = \vartheta(x,A,\cdot)$ and  $\vartheta_h = \vartheta_h(x,A,\cdot)$
\begin{align*}
&\left| Q^{2,\gamma}(x,A) - Q^{2,\gamma}_h(x,A) \right|\\
%
%&=\left|\int_{\Yspace}Q^3(x,y',\iota(y_3 A) + \nabla_\gamma \vartheta)\d y  -\sum_{\hcell\in\hmesh} |\hcell|\sum_{i=1}^k\hQuadW_iQ^3(x,\hQuadPi',\iota(\hQuadP_{\hcell,3,i} A) + \nabla_\gamma \vartheta_h(\hQuadPi))\right| \\
%
& \leq \left|\int_{\Yspace}Q^3(x,y',\iota(y_3 A) + \nabla_\gamma \vartheta(y))-Q^3(x,y',\iota(y_3 A) + \nabla_\gamma \vartheta_h(y))\d y \right|\\ 
& \quad + \left| \int_{\Yspace}Q^3(x,y',\iota(y_3 A) + \nabla_\gamma \vartheta_h(y))\d y-\sum_{\hcell\in\hmesh} |\hcell|\sum_{i=1}^k\hQuadW_iQ^3(x,\hQuadPi',\iota(\hQuadP_{\hcell,3,i} A) + \nabla_\gamma \vartheta_h(\hQuadPi))\right|\\
\end{align*}
The first term on the right-hand side can be estimated as follows
\begin{align*}
&\left|\int_{\Yspace}Q^3(x,y',\iota(y_3 A) + \nabla_\gamma \vartheta + ( \nabla_\gamma \vartheta_h -  \nabla_\gamma \vartheta) ) -Q^3(x,y',\iota(y_3 A) + \nabla_\gamma \vartheta)\d y \right|\\
&= \left|\int_{\Yspace}2 \mathcal{C}^3(x,y') (\iota(y_3A) \!+\! \nabla_\gamma \vartheta)\!:\!(\nabla_\gamma \vartheta_h \!-\!  \nabla_\gamma \vartheta) 
\!+\! \mathcal{C}^3(x,y') (\nabla_\gamma \vartheta_h \!-\! \nabla_\gamma \vartheta)\!:\!(\nabla_\gamma \vartheta_h \!-\!  \nabla_\gamma \vartheta) \d y \right|\\ 
&= 0 + \left|\int_{\Yspace} \mathcal{C}^3(x,y') (\nabla_\gamma \vartheta_h - \nabla \vartheta):(\nabla_\gamma \vartheta_h -  \nabla_\gamma \vartheta) \d y \right|\\
& \leq C \norm{\vartheta_h(x,A,\cdot)- \vartheta(x,A,\cdot)}^2_{W^{1,2}(\Yspace;\R^3)} \leq C h^2 \vert A \vert^2\,.
\end{align*}
Here,  the second equality follows from the Euler-Lagrange equation \eqref{eq:ELeq}.
Applying the general theory for quadrature errors in \cite{Ci78} and taking into account the regularity of $\mathcal{C}^3$, the second term can be estimated by 
$C h \Vert \mathcal{C}^3\Vert_{L^\infty(\omega, W^{1,\infty}(Y, \R^{2\times2\times2\times2}))} \vert A \vert^2$, which is bounded by  $C h \vert A \vert^2$, 
and vanishes under the consistency assumption \eqref{eq:ConsistencyCondition}. This implies the claim.
\end{proof}
The assumption $\mathcal{C}^3\in W^{1,\infty}(\omega\times Y';\R^{2\times2\times2\times2})$ is crucial for this estimate: the Lipschitz continuity in the microscopic variable $y'$ ensures  a first order error estimate with respect to the microscopic grid size $h$ given the microscopic discretization for a given macroscopic $x\in \omega$. The Lipschitz continuity in the macroscopic variable $x$ then implies a uniform error bound of the macroscopic approximation error, later used in the actual two-scale error estimation. This regularity assumptions match with the assumptions in the literature on numerical homogenization for elliptic problems on volumetric domains, \cf \cite{Ab05}, \cite{Ab06}.

In typical engineering applications the material parameters are 
frequently piecewise constant on the microscale. In this less regular case one would loose the first order error estimate for the microscopic approximation error. 
This would prevent us from controlling the microscopic finite element error needed in the $\liminf$- estimate and in the recovery sequence estimate 
of our main result on $\Gamma$-convergence in \cref{thm:main}.
%Taking into account convolution based approximations of these discontinuous parameter functions would be at the price of an additional factor in the estimates scaling with the inverse of the width of the convolution kernel.

%%%%%%%%%%%%%%%%%%%%%%%%%%%%%%%%%%%%%%%%%%%%%%%%%%%%%%%%%%%%%%
%%%%%%%%%%%%%%%%%%%%%%%%%%%%%%%%%%%%%%%%%%%%%%%%%%%%%%%%%%%%%%
%%%%%%%%%%%%%%%%%%%%%%%%%%%%%%%%%%%%%%%%%%%%%%%%%%%%%%%%%%%%%%
\section{Discretization of the macroscopic problem}\label{sec:DKTapprox}
In this section, we numerically approximate the energy \eqref{eq:fullfullEnergy}. 
To this end, we will derive a non-conforming finite element discretization for the two-scale problem and the corresponding discrete
isometry constraint for the macroscopic deformation.
First, let us review the non-conforming finite element approximation based on the Discrete Kirchhoff Triangle (DKT).
For simplicity, we directly assume that $\omega$ is a polygonal domain and $\Gamma_D$ is a union of edges of its boundary.
Let $\Hmesh$ be a regular triangulation of $\omega$ with maximal triangle diameter $H>0$.
We denote by $\mathcal{N}_H$ the set of vertices and by $\mathcal{E}_H$ the set of edges.
For $k\in \mathbb{N}$, we denote by $\mathcal{P}_k$ the set of polynomials of degree at most $k$.
For vertices $z_1,z_2,z_3\in \mathcal{N}_H$ of a triangle $\Hcell$ we define $z_\Hcell = (z_1+z_2+z_3)/3$ as the center of mass of $\Hcell$
and introduce the reduced space of cubic polynomials
\begin{align*}
\mathcal{P}_{3,\text{red}}(\Hcell) \coloneqq \left\{ p\in \mathcal{P}_3(\Hcell) \; \Big\vert \; 6p(z_\Hcell)
= \sum_{i = 1,2,3}\left( 2p(z_i) - \nabla p(z_i)\cdot (z_i - z_\Hcell) \right) \right\}
\end{align*}
which still has $\mathcal{P}_{2}(\Hcell)$ as a subspace and the finite element spaces
\begin{align*}
{\bf{W}}_H \coloneqq& \left\{ w_H \in C(\bar{\omega}) \; \vert \; w_H \restr{\Hcell} \in \mathcal{P}_{3,\text{red}}(\Hcell) \text{ for all }\Hcell \in \Hmesh  \text{ and }\nabla w_H \text{ is continuous at }\mathcal{N}_H \right\}\,, \\
{\bf{\Theta}}_H \coloneqq& \left\{ \theta_H \in C(\bar{\omega};\R^2) \; \vert \; \theta_H\restr{\Hcell} \in \mathcal{P}_2(\Hcell)^2 \text{ and }\theta_H \cdot n_E\restr{E} \text{ is affine for all } E\in \mathcal{E}_h  \right\}\, .
\end{align*}
For a function $w \in W^{3,2}(\omega)$, the interpolation $w_H = \mathcal{I}^{DKT} w \in {\bf{W}}_H$ is defined on every triangle $\Hcell\in \Hmesh$ by
$w_H(z) = w(z)$ and $\nabla w_H(z) = \nabla w(z)$ for all vertices $z\in \mathcal{N}_H \cap \Hcell$,
which is well-defined due to the continuous embedding of $W^{3,2}(\omega)$ into $C^1(\bar \omega)$.
The discrete gradient operator $\theta_H \colon {\bf{W}}_H \to {\bf{\Theta}}_H$ is defined via
\begin{align*}
\theta_H[w_H](z) = \nabla w_H(z) \, ,  \quad \theta_H[ w_H](z_E)\cdot \tau_E = \nabla w_H(z_E) \cdot \tau_E
\end{align*}
for all vertices $z \in \mathcal{N}_H$, all edges $E\in \mathcal{E}_H$ with $\tau_E$ denoting a unit tangent vector on $E$, and $z_E$ the midpoint of $E$.
We use superscripts $(\theta^j_H[w_H])_{j=1,2}$
to indicate the components of $\theta_H[w_H]$.
The operator $\theta_H$ can analogously be defined on $W^{3,2}(\omega)$.
This operator has the following properties (cf.~\cite{Ba11}):

There exists constants $c_0,\,c_1,\,c_2, \,c_3 > 0$ such that for $\Hcell \in \Hmesh$ with $H=\mathrm{diam(\Hcell)}$,
$w\in W^{3,2}(\Hcell)$ and $w_H \in {\bf{W}}_H$
\begin{align}
\label{eq:interpolDKT}
&\Vert w-\mathcal{I}^{DKT} w \Vert_{W^{m,2}(\Hcell)}  \leq c_0 H^{3-m} \Vert w\Vert_{W^{3,2}(\Hcell)} \quad \text{for } m = 0,1,2,3 \,,\\
\label{eq:ia}
& c_1^{-1} \lznormT{D^{k+1}w_H} \leq \lznormT{D^k\theta_H[ w_H]} \leq c_1 \lznormT{D^{k+1}w_H} \quad \text{for } k = 0,1\,, \\
\label{eq:ib}
&\lznormT{\theta_H[ w_H] - \nabla w_H} \leq c_2 H \lznormT{D^2 w_H}\,,\\
\label{eq:ic}
&\lznormT{\theta_H[ w] - \nabla w} + H \lznormT{\nabla \theta_H[ w] - D^2 w} \leq c_3 H_T^2 \norm{w}_{W^{3,2}(\Hcell)}\,.
\end{align}
Furthermore, the mapping $w_H \mapsto \lznormO{\nabla \theta_H[ w_H]}$ defines a norm on
$$\left\{ w_H \in {\bf{W}}_H  \; \vert \; \ w_H(z)=0,\ \nabla w_H(z) = 0 \text{ for all }z \in \mathcal{N}_H \cap \Gamma_D \right\}\,.$$
In our case of macroscopically varying microstructures we have to take into account the 
associated macroscopic quadrature error.
In fact, on every $\Hcell\in\Hmesh$ we use a
numerical quadrature scheme with quadrature points $\HQuadPi\in T$, $i = 1\dots l$ and weights $\HQuadW_i >0$, which is exact
on quadratic polynomials. 
The quadrature scheme is supposed to be defined on a reference triangle 
and transferred by the affine reference map to the actual triangle. 
Then, the associated discrete homogenized total free energy is given by
\begin{align}
\label{eq:discrE}
E^h_H[\psi_H] =\begin{cases} \sum_{\Hcell\in\Hmesh} \!|\Hcell|\sum_{i=1}^l\HQuadW_i 
\left(Q^{2,\gamma}_h(\HQuadPi, \nabla \theta_H[\psi_H](\HQuadPi)) \! -\! f(\HQuadPi) \cdot \psi_H(\HQuadPi)\right)\,,
\quad &\!\!\!\text{if }\psi_H \in \mathcal{A}^\text{BC}_H \\
+\infty\,,&\!\!\!\text{ else},
\end{cases}
\end{align}
with
\begin{align} \nonumber
\mathcal{A}^\text{BC}_H =\Big\{ \psi_H\in {\bf{W}}_H^3\, \Big\vert& \, \nabla \psi_H(z)^\top \nabla \psi_H(z) = I_2 \; \forall z\in \mathcal{N}_H; \; \\
&\psi_H(z)=\phi(z), \nabla \psi_H(z) = \nabla\phi(z)\; \forall z \in  \mathcal{N}_H \cap \Gamma_D \Big\}\,.
\label{eq:ABCH} 
\end{align}
Hence, the isometry property is only enforced on the nodes of the triangulation.
\bigskip 

We now state the main convergence result.
\begin{theorem}
Assume $(\mathcal{T}_H)_H$ is a family of regular triangulation of $\omega$
with grid size $H \rightarrow 0$, the microscopic quadrature rule is exact on the set of tri-affine functions and unisolvent with respect to this set,  the macroscopic quadrature rule is 
exact on quadratic polynomials, and 
$\mathcal{C}^3 \in W^{1,\infty}(\omega \times Y,\R^{3\times 3 \times 3 \times 3})$,  
$f \in W^{1,p}(\omega,\R^3)$, with $p > 2$.
Then, for every macroscopic grid size $H$ and every microscopic grid size  $h>0$ there exists a 
minimizer $\psi_H^{h}$ of the discrete homogenized total free energy $E^h_H[\cdot]$ (cf.~\eqref{eq:discrE}) in $\mathcal{A}^\text{BC}_H$. 
For any sequence $(\psi_H^{h})_{H,h}$ of such minimizers with $H,\; h  \rightarrow 0$ there is a subsequence, such that
with out reindexing $\psi_H^{h} \rightarrow \psi$ strongly in $W^{1,2}(\omega;\R^3)$ . 
Furthermore, $\psi\in \mathcal{A}^\text{BC}$ is a minimizer of the 
continuous total free genergy $E[\cdot]$ as defined in \eqref{eq:fullfullEnergy}.
\label{thm:main}
\end{theorem}
\begin{proof}
\allowdisplaybreaks
The proof combines $\Gamma$-convergence arguments related to those used in \cite{Ba11} and the paradigm 
of the heterogeneous multiscale method as it is discussed in \cite{EMiZh05}.

From the boundedness of $Q^{2,\gamma}_h$ stated in \eqref{eq:alphaBetah}, 
the clamped boundary conditions, the Poincar{\'e} inequality, and the assumption 
on the macroscopic quadrature scheme, we deduce that there exist constants 
$C,\;c >0$ such that
$E_H^{h}[\psi_H^{h}] \geq c \norm{\nabla \theta_H[\psi_H^{h}]}_{L^2(\omega)}^2 - C\,.$ 

Since $\lznormO{\nabla \theta_H[\cdot]}$ is a norm on 
$\left\{ w_H \in {\bf{W}}_H  \vert  \ w_H(z)\!=\!0,\ \nabla w_H(z)\! =\! 0 \; \forall z \in \mathcal{N}_H \cap \Gamma_D \right\}\,, $  
for every $H,h>0$ there exists a minimizer $\psi_H^{h}$ for the discrete minimization problem and using \eqref{eq:ia} we have that 
$\lznormO{\nabla \theta_H[\psi_H^{h}]} + \lznormO{\nabla \psi_H^{h}}\leq C\,.$ 
Thus, there exists a subsequence (not relabeled), and $\psi \in W^{1,2}(\omega;\R^3)$, $z \in W^{1,2}(\omega; \R^{3\times2})$, such that $\psi_H^{h} \rightharpoonup \psi$ in $W^{1,2}(\omega;\R^3)$ and $\theta_H[\psi_H^{h}] \rightharpoonup z$ in $W^{1,2}(\omega; \R^{3\times2})$. Using the estimate \eqref{eq:ib} and summing over all $T\in\mathcal{T}_H$, we get $\lznormO{\nabla \psi_H^{h} - \theta_H[\psi_H^{h}]} \leq cH\lznormO{\nabla \theta_H[\psi_H^{h}]}$, hence $\nabla \psi_H^{h} \rightarrow \nabla \psi = z$ strongly in $L^2$. This yields $\psi \in W^{2,2}(\omega; \R^3)$. As in \cite{Ba11}, the attainment of the boundary conditions and the isometry constraint is straightforward. Hence, $\psi$ is admissible. 
Next, we investigate the convergence of the discrete force term.
\begin{align}\label{eq:quadratureErrorRHS}
\left\vert \sum_{\Hcell\in\Hmesh} |\Hcell|\sum_{i=1}^l\HQuadW_i f(\HQuadPi) \cdot \psi_H^h(\HQuadPi)-
\int_\omega f\cdot \psi_H^h\d x\right\vert \leq  C H \norm{f}_{W^{1,p}(\omega)} \norm{\psi_H^h}_{W^{1,2}(\omega)}
\end{align}
follows from the approximation properties of the quadrature rule, the regularity of $f$, and the estimate for the corresponding force term 
quadrature error in \cite[Theorem 4.1.5, Theorem 4.1.6 and the corresponding estimate in its proof]{Ci78}. 
By the boundedness of $\psi_H^{h}$ in $W^{1,2}(\omega;\R^3)$ the right hand side vanishes for $H\rightarrow 0$.
Furthermore, using the estimate for the quadratic form quadrature error in \cite[Theorem 4.1.4]{Ci78} we obtain
\begin{align}
\label{eq:quadratureError}
\left||\Hcell|\sum_{i=1}^l\HQuadW_i Q^{2,\gamma}(\HQuadPi, \nabla \theta_H[\psi_H^{h}](\HQuadPi)) - \int_\Hcell Q^{2,\gamma}(x, \nabla \theta_H[\psi_H^{h}] ) \d x\right| \leq C H \norm{\nabla \theta_H[\psi_H^{h}]}_{L^2(\Hcell)}^2\,,
\end{align}
since $Q^{2,\gamma}(\cdot, A)$ arising from the microscopic problem is Lipschitz continuous in $x$ and quadratic in $A\in\R^{2\times2}_\text{sym}$, and $\nabla\theta[\psi_H^{h}]$ is affine 
on every $\Hcell\in \Hmesh$. In particular, the constant on the right-hand side depends on the $W^{1,\infty}$-norm of the microscopic elasticity tensor $\mathcal{C}^{3}$. Next, we estimate
\begin{align}
&\sum_{\Hcell\in\Hmesh} |\Hcell|\sum_{i=1}^l\HQuadW_i Q^{2,\gamma}_{h}(\HQuadPi, \nabla \theta_H[\psi_H^{h}](\HQuadPi))
= \sum_{\Hcell\in\Hmesh} |\Hcell|\sum_{i=1}^l\HQuadW_i Q^{2,\gamma}(\HQuadPi, \nabla \theta_H[\psi_H^{h}](\HQuadPi))\nonumber\\
 &\quad \qquad\qquad+\sum_{\Hcell\in\Hmesh} |\Hcell|\sum_{i=1}^l\HQuadW_i \left(Q^{2,\gamma}_{h}(\HQuadPi, \nabla \theta_H[\psi_H^{h}](\HQuadPi)) 
- Q^{2,\gamma}(\HQuadPi, \nabla \theta_H[\psi_H^{h}](\HQuadPi))\right)\nonumber\\ 
&\geq \sum_{\Hcell\in\Hmesh} |\Hcell|\sum_{i=1}^l\HQuadW_i Q^{2,\gamma}(\HQuadPi, \nabla \theta_H[\psi_H^{h}](\HQuadPi)) - C h \sum_{\Hcell\in\Hmesh} |\Hcell|\sum_{i=1}^l\HQuadW_i  \vert \nabla \theta_H[\psi_H^{h}](\HQuadPi)\vert^2 \nonumber\\
&= \int_\omega Q^{2,\gamma}(x, \nabla \theta_H[\psi_H^{h}] ) \d x - C h  \norm{\nabla\theta_H[\psi_H^{h}]}_{L^2(\omega)}^2\nonumber\\
&\quad - \sum_{\Hcell\in\Hmesh}\left( |\Hcell|\sum_{i=1}^l\HQuadW_i Q^{2,\gamma}(\HQuadPi, \nabla \theta_H[\psi_H^{h}](\HQuadPi)) 
- \int_\Hcell Q^{2,\gamma}(x, \nabla \theta_H[\psi_H^{h}] ) \d x  \right) \nonumber\\
&\geq\int_\omega Q^{2,\gamma}(x, \nabla \theta_H[\psi_H^{h}] ) \d x -  C(h+H ) \norm{\nabla\theta_H[\psi_H^{h}]}_{L^2(\omega)}^2\,. \label{eq:TotalErr}
\end{align}

Here, in the first inequality, we used the error estimate \eqref{eq:ErrorHMM}. Furthermore, we used that $\vert\nabla\theta_H[\psi_H^{h}]\vert^2$ is a quadratic polynomial and thus can be integrated exactly by the macroscopic quadrature scheme. The second inequality is an application of \eqref{eq:quadratureError}. Since $\nabla \theta_H[\psi_H^{h}]$ converges weakly to $D^2\psi$ in $L^2(\omega,\R^{3\times 2\times 2})$ and $(x,A) \mapsto Q^{2,\gamma}(x,A)$ is 
Lipschitz continuous in $x$, and convex and quadratic in $A$, 
we deduce by weak lower semicontinuity that $E[\psi] \leq \liminf_{H,h \rightarrow 0} E^{h}_H[\psi_H^{h}]$.
\bigskip

Regarding a recovery sequence, let $\psi^\ast \in \mathcal{A}$ be a minimizer of $E[\cdot]$, in particular $\psi^\ast\in W^{2,2}(\omega;\R^3)$. To make use of the estimate \eqref{eq:ic}, $W^{3,2}$-regularity is required. Hence, a density result together with a diagonal sequence argument would suffice to get convergence. Standard mollification of $\psi^\ast$ is not suitable, since it violates the isometry constraint. However, Hornung showed in \cite[Theorem 1]{Ho11} that for $\rho > 0$ there exists $\psi^\rho \in W^{3,2}(\omega; \R^3) \cap \mathcal{A}$ such that $\norm{\psi^\ast - \psi^\rho}_{W^{2,2}} < \rho$. A version of this density result was first proved by Pakzad for isometric immersions $\psi\colon\omega\rightarrow\R^3$ of convex sets $\omega\subset\R^2$ in \cite[Theorem I]{Pa04}.  Let $\psi_H^\rho \in {\bf{W}}_H^3$ be the DKT-interpolant of $\psi^\rho$. 
The same construction of a recovery sequence for DKT was already used in \cite{Ba11, BaBoNo17, RuSiSm22}.
Then we can estimate 

\begin{align*}
&\sum_{\Hcell\in\Hmesh} |\Hcell|\sum_{i=1}^l\HQuadW_i Q^{2,\gamma}_{h}(\HQuadPi, \nabla \theta_H[\psi_H^\rho](\HQuadPi))
= \sum_{\Hcell\in\Hmesh} |\Hcell|\sum_{i=1}^l\HQuadW_i Q^{2,\gamma}(\HQuadPi, \nabla \theta_H[\psi_H^\rho](\HQuadPi))\\
&\qquad \qquad \qquad \qquad +\sum_{\Hcell\in\Hmesh} |\Hcell|\sum_{i=1}^l\HQuadW_i \left(Q^{2,\gamma}_{h}(\HQuadPi, \nabla \theta_H[\psi_H^\rho](\HQuadPi)) - Q^{2,\gamma}(\HQuadPi, \nabla \theta_H[\psi_H^\rho](\HQuadPi))\right)\\
&\leq\sum_{\Hcell\in\Hmesh} |\Hcell|\sum_{i=1}^l\HQuadW_i Q^{2,\gamma}(\HQuadPi, \nabla \theta_H[\psi_H^\rho](\HQuadPi)) 
+ Ch \sum_{\Hcell\in\Hmesh} |\Hcell|\sum_{i=1}^l\HQuadW_i  \vert \nabla \theta_H[\psi_H^{h}](\HQuadPi)\vert^2 \\
&=\int_\omega Q^{2,\gamma}(x, \nabla \theta_H[\psi_H^\rho] ) \d x+ C h \norm{\nabla \theta_H[\psi_H^\rho]}_{L^2(\omega)}^2\\
&\quad + \sum_{\Hcell\in\Hmesh}\left( |\Hcell|\sum_{i=1}^l\HQuadW_i Q^{2,\gamma}(\HQuadPi, \nabla \theta_H[\psi_H^\rho](\HQuadPi)) - \int_\Hcell Q^{2,\gamma}(x, \nabla \theta_H[\psi_H^\rho] ) \d x \right)\\
&\leq\int_\omega Q^{2,\gamma}(x, \nabla \theta_H[\psi_H^\rho] ) \d x + C(H+h) \norm{\nabla \theta_H[\psi_H^\rho]}_{L^2(\omega)}^2   \\
&\leq \int_\omega Q^{2,\gamma}(x, D^2\psi^\rho) + \mathcal{C}^{2,\gamma}  
\left(D^2\psi^\rho + \nabla \theta_H[\psi_H^\rho] \right) : 
\left(D^2\psi^\rho - \nabla \theta_H[\psi_H^\rho]\right) \d x  \\
&\quad  + C(H+h) \norm{\nabla \theta_H[\psi_H^\rho]}_{L^2(\omega)}^2\\
&\leq\int_\omega Q^{2,\gamma}(x, D^2\psi^\rho) \d x+ C \norm{D^2\psi^\rho + \nabla \theta_H[\psi_H^\rho]}_{L^2(\omega)}\norm{D^2\psi^\rho - \nabla \theta_H[\psi_H^\rho]}_{L^2(\omega)}\\
&\quad + C(H+h) \left(\norm{\psi^\rho}_{W^{2,2}(\omega)}^2+ H^2\norm{\psi^\rho}_{W^{3,2}(\omega)}^2\right)\\
&\leq \int_\omega Q^{2,\gamma}(x, D^2\psi^\ast) \d x + C \norm{D^2 \psi^\ast + D^2\psi^\rho }_{L^2(\omega)}\norm{D^2\psi^\ast - D^2\psi^\rho }_{L^2(\omega)}\\
&\quad  + C H \norm{\psi^\rho}_{W^{3,2}(\omega)} \left(  \norm{\psi^\rho}_{W^{2,2}(\omega)}+ H\norm{\psi^\rho}_{W^{3,2}(\omega)}\right)\\
&\quad + C(H + h)\left(  \norm{\psi^\rho}_{W^{2,2}(\omega)}^2+ H^2\norm{\psi^\rho}_{W^{3,2}(\omega)}^2\right)\\
&\leq \int_\omega Q^{2,\gamma}(x, D^2\psi^\ast) \d x 
+ C (1+\rho) \rho  
+ C H \norm{\psi^\rho}_{W^{3,2}(\omega)} \left(1+ \rho +  H \norm{\psi^\rho}_{W^{3,2}(\omega)}\right) \\
&\quad + C (H+h)\left( (1+\rho)^2 + H^2 \norm{\psi^\rho}_{W^{3,2}(\omega)}^2\right)\,.
\end{align*}
Here, we again used \eqref{eq:ErrorHMM},  \eqref{eq:quadratureError}, 
the estimate 
$$
\norm{\nabla \theta_H[\psi_H^\rho]}_{L^2(\omega)} \!\leq \!\norm{D^2\psi^\rho}_{L^2(\omega)}
+\norm{\nabla \theta_H[\psi_H^\rho]-D^2\psi^\rho}_{L^2(\omega)}
\!\leq\! \norm{\psi^\rho}_{W^{2,2}(\omega)}
+ H \norm{\psi^\rho}_{W^{3,2}(\omega)}
$$
as a consequence of \eqref{eq:ic}, and finally the estimate $\norm{\psi^\ast - \psi^\rho}_{W^{2,2}} < \rho$.
For the force term, we get using the estimate for the linear form quadrature error in \cite[Theorem 4.1.5]{Ci78}
\begin{align*}
&\sum_{\Hcell\in\Hmesh} |\Hcell|\sum_{i=1}^l\HQuadW_i f(\HQuadPi) \cdot \psi_H^\rho(\HQuadPi)
\leq \int_\omega f\cdot \psi^\ast \d x + \int_\omega f\cdot ( \psi_H^\rho - \psi^\ast) \d x \\ 
& \qquad \qquad \qquad \qquad 
+ \left(\sum_{\Hcell\in\Hmesh} |\Hcell|\sum_{i=1}^l\HQuadW_i f(\HQuadPi) \cdot \psi_H^\rho(\HQuadPi) - \int_\omega f\cdot \psi_H^\rho \d x \right)\\
&\leq \int_\omega f\cdot \psi^\ast \d x + \norm{f}_{L^2(\omega)} 
\left( \norm{\psi_H^\rho- \psi^\rho}_{L^2(\omega)} +  \norm{\psi^\rho- \psi^\ast}_{L^2(\omega)} \right) \\
&\quad + C H \norm{f}_{W^{1,p}(\omega)}  \norm{\psi_H^\rho}_{W^{1,2}(\omega)} \\
&\leq \int_\omega f\cdot \psi^\ast \d x + C \left(H^3 \norm{\psi^\rho}_{W^{3,2}(\omega)} + \rho\right)
+ C H (H^2 +1) \norm{\psi^\rho}_{W^{3,2}(\omega)}
\end{align*}
Finally, we choose $h=h(\rho)$ and $H = H(\rho)$ such that 
\begin{align}
\label{eq:coupling}
h(\rho) + H(\rho) + \left(h(\rho)+1\right)\left(H(\rho) 
+H(\rho)^3\right)\left(\norm{\psi^\rho}_{W^{3,2}(\omega)} 
+ \norm{\psi^\rho}^2_{W^{3,2}(\omega)}\right) < \rho
\end{align}
and obtain 
\begin{align*}
E[\psi^\ast] \geq \limsup_{\rho\rightarrow 0 } 
E^{h(\rho)}_{H(\rho)}[\psi_{H(\rho)}^\rho] 
\geq \liminf_{\rho\rightarrow 0} E^{h(\rho)}_{H(\rho)}[\psi_{H(\rho)}^{h(\rho)}] \geq E[\psi]\,.
\end{align*}

Hence, $\psi$ is a minimizer of $E[\cdot]$.
\end{proof}

\begin{remark}
		The convergence of the recovery sequence is crucially controlled by the approximation parameter $\rho$ via the appearance of the 
		norm $\norm{\psi^\rho}_{W^{3,2}(\omega)}$. 
		The choice of $h=h(\rho)$ and $H = H(\rho)$ in \eqref{eq:coupling} is implicit and thus does not provide a rate of convergence. 
		This disadvantage is shared with similar results in \cite{Ba11} and \cite{BaBoNo17} 
		and is caused by the nonlinearity of the problem, more precisely the nonlinear isometry constraint.	
\end{remark}

\begin{remark}
	Note that the regularity assumption on the force $f\in W^{1,p}(\omega;\R^3)$ with $p>2$ is necessary to control the force term integration error. To this end, as outlined in the proof, one uses \cite[Theorem 4.1.5 and Theorem 4.1.6]{Ci78}. 
		This assumption is not a peculiarity of the homogenized plate model. In fact, it would also be necessary for the numerical analysis of a 
		finite element approximation of a homogeneous plate model as in \cite{Ba11} with general right hand side.
\end{remark}

%%%%%%%%%%%%%%%%%%%%%%%%%%%%%%%%%%%%%%%%%%%%%%%%%%%%%%%%%%%%%%
%%%%%%%%%%%%%%%%%%%%%%%%%%%%%%%%%%%%%%%%%%%%%%%%%%%%%%%%%%%%%%
%%%%%%%%%%%%%%%%%%%%%%%%%%%%%%%%%%%%%%%%%%%%%%%%%%%%%%%%%%%%%%
\section{Implementational aspects}\label{sec:Implementation}
Following the paradigm of the heterogeneous multiscale method (HMM) the actual macroscopic plate deformation is computed using the 
DKT discretization described in ~\cref{sec:DKTapprox} and minimizing the discrete energy 
$E^h_H$ defined in \eqref{eq:discrE}. To this end for each macroscopic quadrature 
the function $Q^{2,\gamma}_h(A)$ has to be computed for a basis of the space of symmetric 
$2\times 2$ matrices $A$, which requires the solution of 
the micrscopic optimization problem \eqref{eq:Q_hinf}.  

In explicit, the microscopic problem results in solving the linear system 
\eqref{eq:Qgamma_hvartheta_h} for $x$ in the set of all macroscopic quadrature points $\HQuadPi$  and for
\begin{align}
\label{eq:BasisSymMatr}
A \in 
\left\{  \begin{pmatrix}
1 & 0 \\ 0 & 0
\end{pmatrix},\begin{pmatrix}
0 & 0 \\ 0 & 1
\end{pmatrix} ,\begin{pmatrix}
0 & 1 \\ 1 & 0
\end{pmatrix}  \right\}\,,
\end{align}
which represents a three-dimensional, linearized elasticity  corrector problem 
rescaled to the unit cube $\Yspace = [0,1]^2\times [-\frac12,\frac12]\subset \R^3$. 
As described in ~\cref{sec:ApproxHomTensor}, continuous, piecewise $3$-affine finite elements are considered on 
the unit cube discretized by a uniform hexahedral mesh $\hmesh$ with $h$ denoting the maximal edge length. For the assembly of the stiffness matrix and the right-hand side we apply a tensor product Gauss quadrature rule, with 27 quadrature points per cell. 
Hence, the quadrature is exact on tensor products of quintic functions, and thus 
fulfills the consistency condition to be exact on and unisolvent with respect to tri-affine functions. This is sufficient to 
ensure a first-order convergence under the assumption that the solution of the linearized elasticity problem is in $W^{2,2}$ on $\Yspace$ 
\cite[Chapter 4.1.]{Ci78} . 
In fact, the strong consistency condition \eqref{eq:ConsistencyCondition} holds if the entries of the elasticity tensor $\mathcal{C}^3_{jikl}$  
are assumed to be tensor products of cubic functions.

On the macroscale, we consider a uniform triangulation of $\omega$ with grid size $H$ and ask for a 
minimizing deformation $\psi^H$ of $E_H^h[\cdot]$ under the nonlinear, discrete isometry constraint 
$\nabla \psi_H(z)^\top \nabla \psi_H(z) = I_2$ for all $z\in \mathcal{N}_H$ (cf.~\eqref{eq:ABCH}).
To compute the discrete energy $E_H^h[\psi_H]$ for $\psi_H \in {\bf{W}}_H^3$ a 
simplicial Gauss quadrature with 12 quadrature points on each triangle is used, which is exact on polynomials of order $6$.
Hence, if the components of the tensor representing the quadratic form $ Q^{2,\gamma}(x,\cdot)$ as a function of $x$ are 
quartic polynomials the integration of the  energy $\mathcal{W}^\gamma$ defined in \eqref{eq:homBendEnergy} is exact.

The macroscopic problem consists of minimizing the energy  $E_H^h[\cdot]$ over all discrete isometries $\phi_H\in\mathcal{A}^\text{BC}_H$.
To deal with the isometry constraint we take into account the Lagrangian 
\begin{align}
L_H^h [\psi_H,p_H] \coloneqq E_H^h[\psi_H] - \int_\omega \mathcal{I}_H\left[ \left(\nabla \psi_H^\top\nabla\psi_H - I_2\right):p_H\right] \d x\,.
\end{align}
with a Lagrangian multiplier $p_H \in \mathcal{S}_H^{2\times2}$, where $ \mathcal{S}_H^{2\times2}$ is the space of continuous piecewise affine, symmetric $2\times2$-matrices. 
Let us remark, that the quadrature scheme applied to the stored elastic energy is exact on the second term of the Lagrangian multiplying the constraint and the multiplier. 
As in \cite{RuSiSm22} we use the IPOPT software library presented in \cite{WaBi06} to compute a saddle point of this Lagrangian
using a Newton scheme with backtracking. This requires evaluating first and second variations of the discrete energy and the constraint. In IPOPT, we used the default backtracking strategy by setting ``filter'' for ``line\_search\_method''. The (relative) stopping tolerance ``tol" was set to $10^{-12}$.
We always used a small perturbation of the identity with $L^\infty$ norm $0.001$ as the initial deformation in the Newton scheme.
Note that an application of Newton's scheme requires the invertibility of the Hessian of the Lagrangian. Although, there is no theoretical guarantee, in our experiments solvability of the associated linear system was always observed.
%%%%%%%%%%%%%%%%%%%%%%%%%%%%%%%%%%%%%%%%%%%%%%%%%%%%%%%%%%%%%%
%%%%%%%%%%%%%%%%%%%%%%%%%%%%%%%%%%%%%%%%%%%%%%%%%%%%%%%%%%%%%%
%%%%%%%%%%%%%%%%%%%%%%%%%%%%%%%%%%%%%%%%%%%%%%%%%%%%%%%%%%%%%%
\section{Numerical experiments}\label{sec:NumEx}
In this section, we numerically verify the established convergence results, 
discuss the qualitative properties of the homogenized plate model,
and compare it with bending experiments for paper with different fine-scale structures.
We always consider a square-shaped plate, described by $\omega = (0,1)^2$. 

We take into account a material distribution function $v$ on a rescaled microscopic cell $Y \times I$  
with Lam{\'e} constants $\lambda(y)$ and $\mu(y)$ given by
\begin{align*}
\lambda(y) = (r + (1 - r)v(y)) \lambda\,, \quad \mu(y) = (r + (1 - r)v(y)) \mu\,.
\end{align*}
where $\lambda = \frac53, \mu = \frac52$, and $r=\frac{1}{50}$ is the ratio between soft and hard material.
This corresponds to Young's modulus $E=6$ and Poisson ratio $\nu = 0.2$ on the hard phase.

The triangulations on $\omega$ are generated by uniform, regular (so-called red) refinement, starting 
from a single, coarse rectangular mesh subdivided into two triangles.
The plate deformations are caused by a constant force $f$ and/or prescribed boundary conditions.
\medskip

\noindent \textit{Experimentally observed convergence behaviour.} 
To examine the convergence rates we consider a homogenized bending energy with a constant microstructure 
given by the continuous, piecewise affine material distribution function 
\begin{align}
\label{eq:pwaffinematerial}
v(y) = \begin{cases}
2 y_1 \,,&\text{ if } y_1 \leq \frac12 \\
-2 y_1 + 2 \,,&\text{ if } y_1 > \frac12\,
\end{cases}
\end{align}
on the unit cell
and a size-of-microstructure-to-thickness ratio of $\gamma = 1$. 
We take into account two different load scenarios:
\medskip
(a) clamped boundary conditions 
\begin{align}
\psi(x) = ( x_1  \pm \tfrac{3}{16}, x_2 , 0)^\top, \quad 
\nabla \psi(x) = \begin{pmatrix} I_2 \\ 0 \ 0 \end{pmatrix}
\label{eq:bdexpA}
\end{align}
on $ (0,1) \times  \{\tfrac12 \mp \tfrac12\}$ and zero load,
(b) a uniform vertical load $f= (0,0,5)^\top$, and clamped boundary conditions 
\begin{align*}
\psi(x) = \begin{pmatrix} x \\ \pm 0 \end{pmatrix}, \quad 
\nabla \psi(x) = \begin{pmatrix} I_2 \\ 0 \ 0 \end{pmatrix}
\end{align*}
on $ (0,1) \times  \{0\}$.
%To avoid a saddle point as starting configuration in (a) we use a small random perturbation with $L^\infty$ norm $0.001$ as initialization for the descent algorithm.
For experiments (a) and (b) the resulting deformations are shown in ~\cref{tab:Convergence_h} on the right. Note that in experiment (a), for different random initial data, either a buckled up, or a buckled down deformation can be observed, as displayed.
In the same table, we evaluate the impact of the microscopic grid size $h$ on the approximation of the effective tensor $\mathcal{C}^{2,\gamma}$.
Since the minimizer of the continuous problem is unknown, we compare it with the discrete solution on the finest mesh ($h^\ast = \frac{1}{128}$)
and list $|\mathcal{C}^{2,\gamma}_h - \mathcal{C}^{2,\gamma}_{h^\ast}|_\infty$ for decreasing $h$. In addition, the $L^2$-norm of the difference of the discrete Hessians $\norm{ \nabla \theta_{H^\ast} [\psi_{H^\ast}^{h}] - \nabla \theta_{H^\ast} [\psi_{H^\ast}^{h^\ast}]}_{L^2}$ is given for a macroscopic grid size $H^\ast = \frac{\sqrt{2}}{256}$. In both cases, we experimentally observe a quadratic convergence rate.
\medskip
%ttttttttttttttt
\begin{table}[h]
\centering
\begin{tabularx}{0.99\textwidth}{c c c c c c}
\toprule
$h$
& $|\mathcal{C}^{2,\gamma}_h - \mathcal{C}^{2,\gamma}_{h^\ast}|_\infty$
&  \multicolumn{2}{c}{$\norm{ \nabla \theta_{H^\ast} [\psi_{H^\ast}^h] - \nabla \theta_{H^\ast} [\psi_{H^\ast}^{h^\ast}]}_{L^2}$}
& \multicolumn{2}{c}{Deformations}\\
\cmidrule(lr){3-4} \cmidrule(lr){5-6} 
&&(a)&(b)& (a) & (b) \\
\midrule
$2^{-3}$&0.004265& 0.005098& 0.0947412& \multirow{5}{*}{\begin{minipage}{0.15\linewidth}\centering
\includegraphics[width=0.9\linewidth]{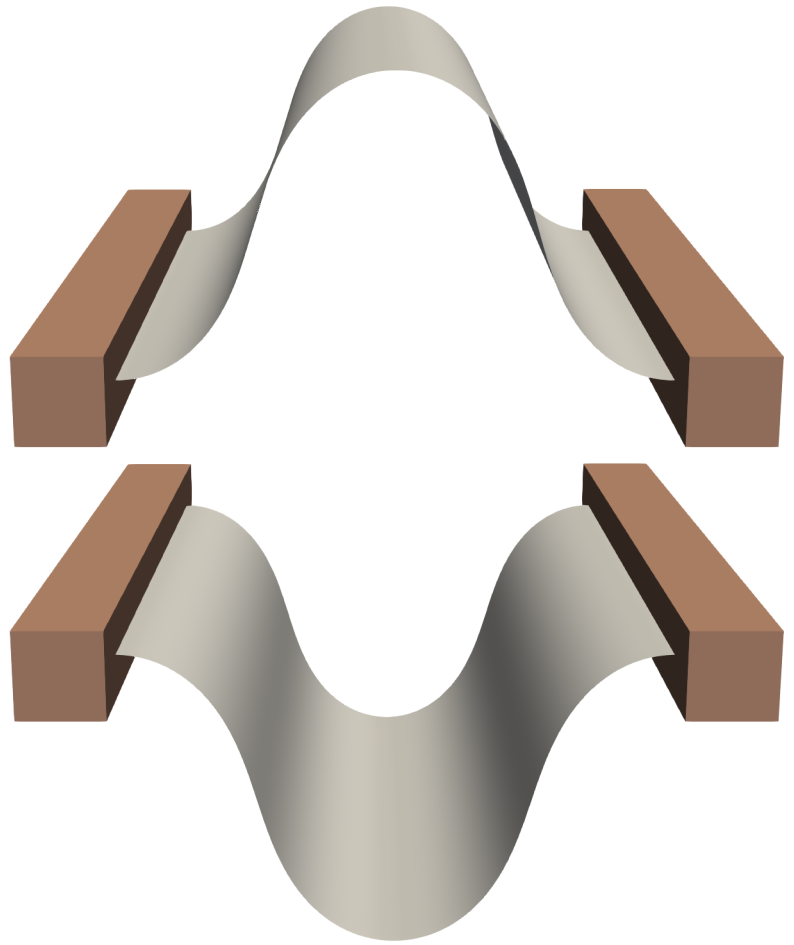}
\end{minipage}}& \multirow{5}{*}{\begin{minipage}{0.085\linewidth}
\includegraphics[width=\linewidth]{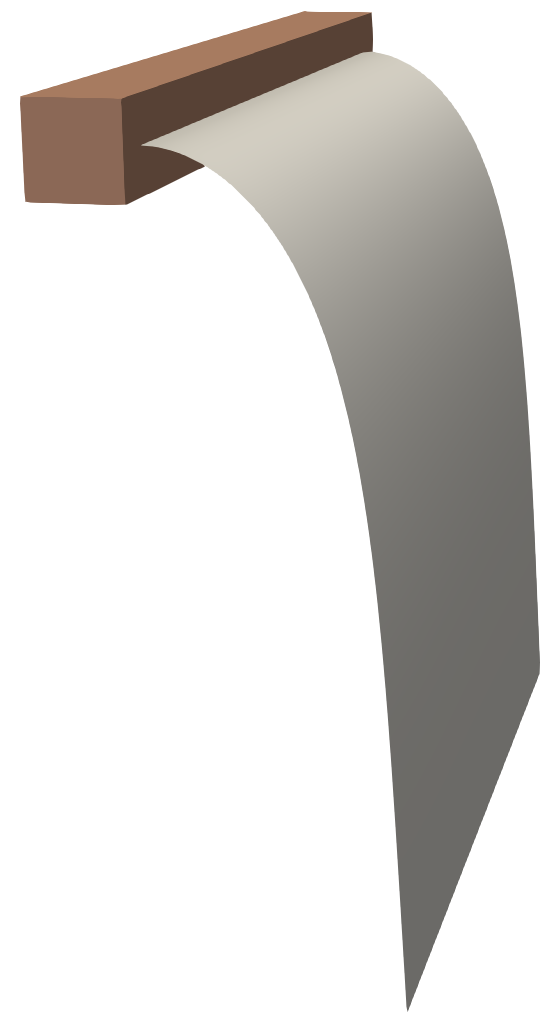}
\end{minipage}} \\
$2^{-4}$&0.001021& 0.001259 & 0.0235653& &\\
$2^{-5}$&0.000243& 0.000311 & 0.00566203& &\\
$2^{-6}$&4.8e-05& 6.4e-05& 0.00112114& &  \\
$2^{-7}$& - & - & - & & \\
\bottomrule
\end{tabularx}
\caption{Evaluation of the experimental convergence of the homogenized tensor $\mathcal{C}^{2,\gamma}_h$
and the discrete Hessian for fixed macroscopic grid size $H^\ast =  \frac{\sqrt{2}}{256}$
and decreasing microscopic grid size $h$.}
\label{tab:Convergence_h}
\end{table}
%ttttttttttttttt

Next, in ~\cref{tab:Convergence_hH} we investigate the convergence behavior for the load configurations (a) and (b) 
in the case of a simultaneous decrease of the microscopic and macroscopic grid size $h$ and $H$, respectively.
The minimal discrete energy $E_H^{h}[\psi_H^{h}(H)]$ and the $L^2$-norm of the difference of the discrete Hessians 
$\norm{ \nabla \theta_H [\psi_H^{h}] - \nabla \theta_{H^\ast} [\psi_{H^\ast}^{h^\ast}]}_{L^2}$ are shown for  
$H^\ast = \frac{ \sqrt{2}}{256}$ and $h^\ast =2H^\ast$. Here, one observes a linear convergence rate.

%ttttttttttttttt
\begin{table}[H]
\centering
\begin{tabularx}{0.9\textwidth}{c c c c c c}
\toprule
$\frac{1}{\sqrt{2}}H$
&
$h$
&  \multicolumn{2}{c}{$E_H^{h}[\psi_H^{h}]$}
&  \multicolumn{2}{c}{$\norm{ \nabla \theta_H [\psi_H^{h}] - \nabla \theta_{H^\ast} [\psi_{H^\ast}^{h^\ast}]}_{L^2}$}\\
\cmidrule(lr){3-4}\cmidrule(l){5-6}
&&(a)&(b)&(a)&(b) \\
\midrule
$2^{-4}$&$2^{-3}$&1.7228944&-1.6769788 &1.17653&0.428218 \\
$2^{-5}$&$2^{-4}$&1.676654&-1.6963124 &0.586753&0.208557  \\
$2^{-6}$&$2^{-5}$&1.6688657&-1.7005727 &0.284853&0.10684 \\
$2^{-7}$&$2^{-6}$&1.6666743&-1.7015843 &0.138362&0.0528297  \\
$2^{-8}$&$2^{-7}$& 1.6664753 & -1.7017899 &-& -\\
\bottomrule
\end{tabularx}
\caption{Evaluation of the experimental convergence of the energy and the Hessian of the macroscopic deformation for simultaneous decreasing macroscopic and microscopic grid size $H$ 
and $h$, respectively.}
\label{tab:Convergence_hH}
\end{table}
%ttttttttttttttt

 Furthermore, we investigate the dependence of $Q^{2,\gamma}$ on the ratio $\gamma=\tfrac{\delta}{\varepsilon}$ 
of the thickness $\delta$ and the size of the microstructure $\varepsilon$ for the same microscopic material distribution \eqref{eq:pwaffinematerial} as before. Using the known symmetries of $\mathcal{C}^{2,\gamma}$ one can consider with a slight misuse of notation
$\mathcal{C}^{2,\gamma,v}\in \R^{3\times3}$ following the Voigt notation for elasticity tensors 
with  
$$Q^{2,\gamma}(A) =  \left(\mathcal{C}^{2,\gamma,v}\begin{pmatrix}
A_{11} \\ A_{22} \\A_{12} + A_{21} 
\end{pmatrix}\right) \cdot \begin{pmatrix}
A_{11} \\ A_{22} \\A_{12} + A_{21} 
\end{pmatrix}\,.$$
Due to the symmetry of the material distribution, we get that $\mathcal{C}^{2,\gamma,v}_{13} = \mathcal{C}^{2,\gamma,v}_{23}= 0 $. 
For $\gamma_i = 10^{(-1 + \frac{2i}{31})}$, $i = 0,\dots,31$, we computed $\mathcal{C}^{2,\gamma_i,v}$ on the unit cube with gridsize $h^\ast = \frac{1}{128}$. The values of the different components of $\mathcal{C}^{2,\gamma,v}$ are plotted in ~\cref{fig:plotGammaTensor} with $\gamma$ on the x-axis plotted with logarithmic scaling.
%ffffffffffffffffff
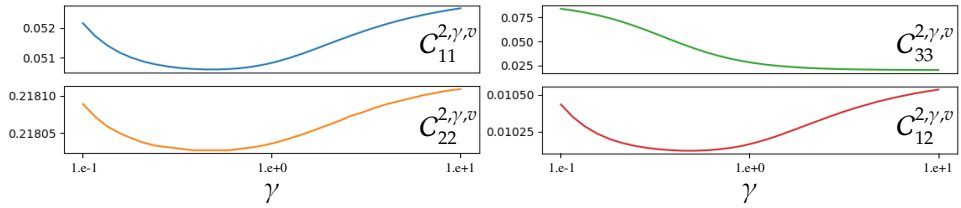
\begin{figure}[H]
\begin{tikzpicture}
\node (0,0) {\includegraphics[width=0.98\linewidth]{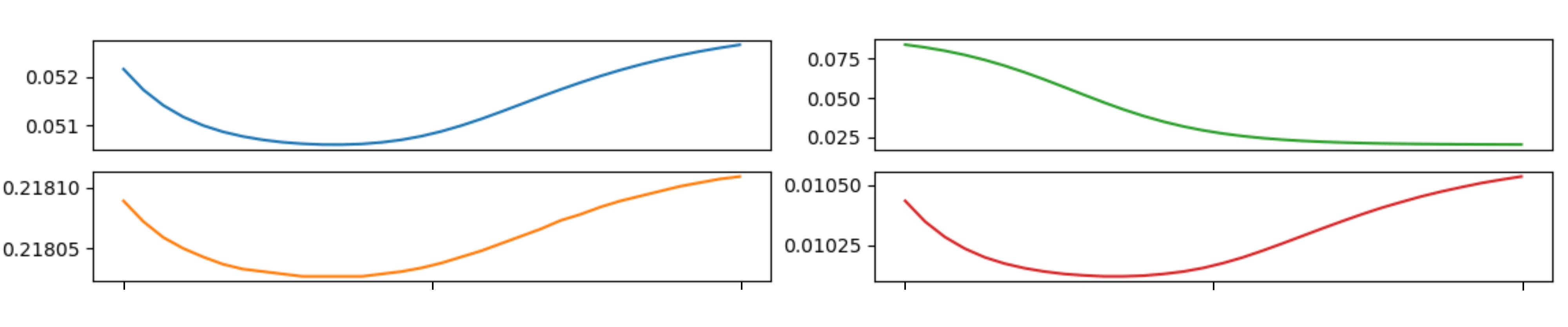}};
\node at (-0.5,0.5) {$\mathcal{C}^{2,\gamma,v}_{11}$};
\node at (-0.5,-0.6) {$\mathcal{C}^{2,\gamma,v}_{22}$};
\node at (5.84,0.5) {$\mathcal{C}^{2,\gamma,v}_{33}$};
\node at (5.84,-0.6) {$\mathcal{C}^{2,\gamma,v}_{12}$};

\node at (-2.85,-1.5) {$\gamma$};
\node at (-0.35,-1.15) {\tiny{1.e+1}};
\node at (-2.85,-1.15) {\tiny{1.e+0}};
\node at (-5.35,-1.15) {\tiny{1.e-1}};

\begin{scope}[xshift=6.35cm]
\node at (-2.85,-1.5) {$\gamma$};
\node at (-0.35,-1.15) {\tiny{1.e+1}};
\node at (-2.85,-1.15) {\tiny{1.e+0}};
\node at (-5.35,-1.15) {\tiny{1.e-1}};
\end{scope}
;\end{tikzpicture}
\caption{Plots of the different components of the homogenized tensor $\mathcal{C}^{2,\gamma}$ for varying $\gamma$.}
\label{fig:plotGammaTensor}
\end{figure}
%ffffffffffffffffff

From now on, we consider microstructures that only consist of a hard and a soft phase with piecewise constant $v:Y\rightarrow [0,1] \to \{0,1\}$. 
Such microstructures do not fulfill the assumptions of our convergence theory. Still, they are closer to real applications and we indeed compare our numerical findings with physical experiments for the bending of fine-scale structured paper. Paper and paperboard are known to deform only isometrically. Note that paper behaves, in opposite to the model we are simulating here, in a non-isotropic way, depending on the direction of manufacturing. In \cite{YoNa07}, in-plane orthotropic constants for paper and paperboard were experimentally derived. 
Nevertheless, for simplicity, we assume in our experiments an isotropic model 
with Lam{\'e} constants $\mu = \frac53$, $\lambda=\frac52$ inspired by the constants derived in \cite{YoNa07}. 
Note that the theory provided in this paper is not limited to isotropic material laws, since \cref{prop:EHMM} is formulated for a more general elastic tensor $\mathcal{C}^3$.

As proof of concept, we carried out experiments with glued layers of paper. 
In explicit, a thick paperboard ($300 g/m^2$) with 
cut out fine-scale structures was glued on a sheet of paper ($120g/m^2$). 
The cutting of the $300 g/m^2$ layer of thick paper was performed with a Cricut\copyright Maker. Regions with glued paper layers are considered the hard phase, 
whereas the complement is considered the soft phase.
Given the fact, that the thickness of the paper layers is much smaller than the scale on which the pattern is approximately periodic 
a model approximation with a homogeneous thick thin plate with vertically constant hard and soft material properties appears to be feasible.

\noindent \textit{Calibration of the layered, glued paper model.} Given the thickness of two paper types one is at first led to a material property ratio of $r = \frac{120}{420} \approx 0.286$. 
But, due to the layer of glue in between the glued paper layers of the hard phase, one observes a proportionally much higher stiffness as examined in ~\cref{fig:GlueComparison}.
To compensate for this we quantitatively compared the bending in the experiment (b) of a single $120g/m^2$ sheet and the homogeneous compound of glued  $300 g/m^2$ and $120g/m^2$ paper, which gave rise to the effective material ratio of $r = \frac{1}{50}$.
We used this ratio in the following experiments.
%ffffffffffffff
\begin{figure}[htbp]
\centering
\includegraphics[width= \linewidth]{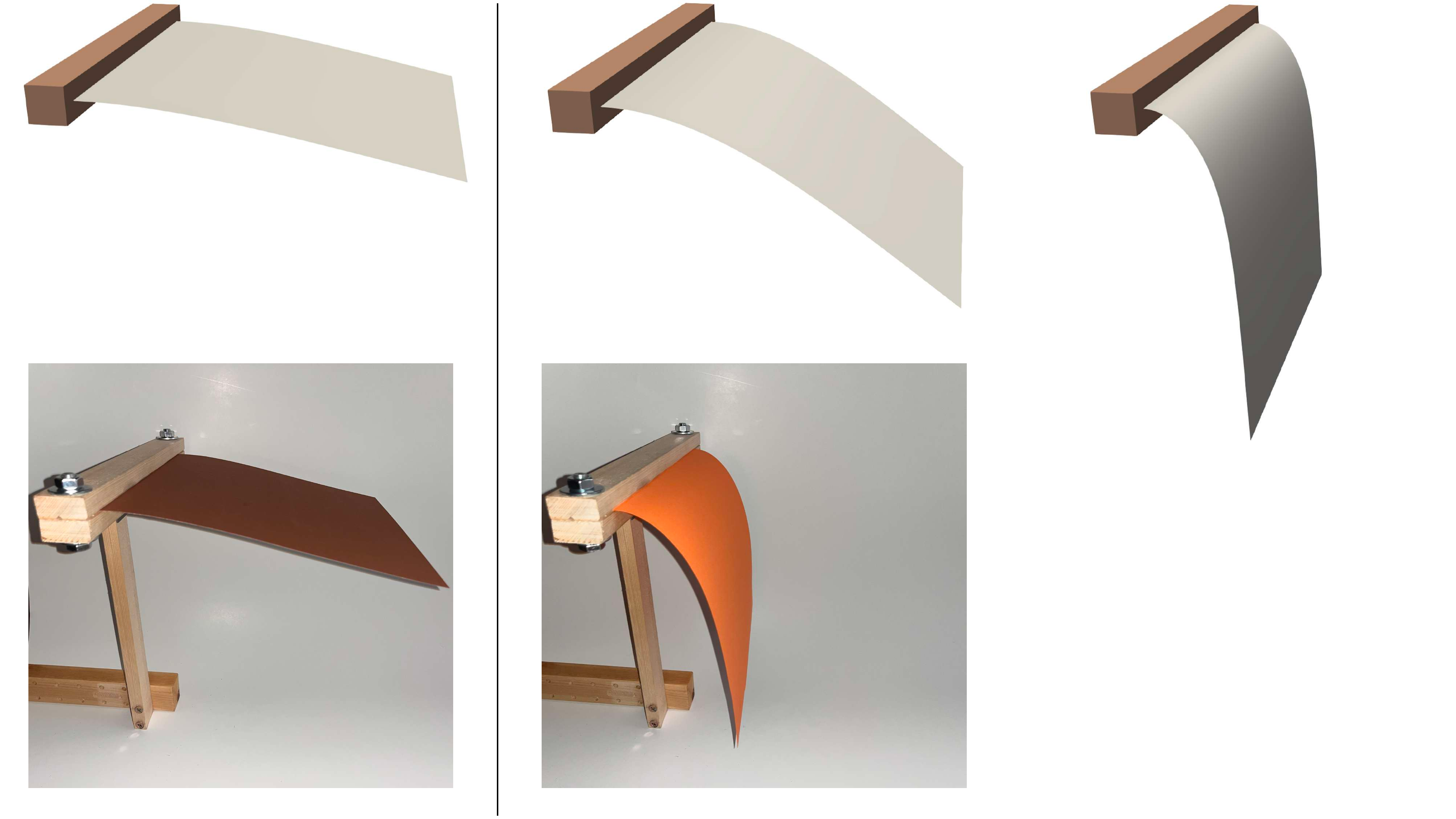}
\caption{Top row: numerically computed deformed configurations of a plate, clamped on the left side, under a uniform vertical load, with homogeneous material, with Lam{\'e} constants scaled with $1.0$ (left), $0.286$ (middle) and $0.02$ (right). Bottom row: photos of physical deformations of paper. Left: thick paper ($300g/m^2$) glued on thinner paper ($ 120g/m^2$). Middle: only thin paper ($120 g/m^2$).}
\label{fig:GlueComparison}
\end{figure}
%ffffffffffffff
\medskip

\noindent \emph{Microstructure with axes aligned stripes.} We consider a material distribution with $\gamma =0.1$ that corresponds to a periodic pattern of 
stripes of alternating hard and soft material aligned with the coordinate directions. On the rescaled microscopic cell $Y \times I$  this 
corresponds to hard material in the volumes $[\frac14,\frac34]\times [0,1]\times [-\frac12,\frac12]$, and soft material elsewhere.
We also take into account the same configuration rotated by $90$ degrees in the $Y$ plane. 
The constant homogenized tensors $\mathcal{C}^{2,\gamma}$ were computed for a uniformly hexagonal mesh on the unit cube with 
grid size $h=\tfrac1{128}$. The resulting tensors $\mathcal{C}^{2,\gamma}_{\vert\vert}$ and $\mathcal{C}^{2,\gamma}_{=}$, 
respectively, written as matrices in Voigt type notation, corresponding to the basis \eqref{eq:BasisSymMatr} are
\begin{align*}
\mathcal{C}^{2,\gamma,v}_{\vert\vert} = \begin{pmatrix}
0.04265 & 0.00853 & 0 \\
0.00853 & 0.5149 & 0 \\
0 & 0 & 0.05099
\end{pmatrix}\,, \quad  \mathcal{C}^{2,\gamma,v}_{=} = \begin{pmatrix}
0.5149 & 0.00853 & 0 \\
0.00853 & 0.04265 & 0 \\
0 & 0 & 0.05099
\end{pmatrix}\,.
\end{align*}
The plate is clamped on $[0,\frac15]\times [\frac25,\frac35]$ and a uniform force $f = (0,0,-5)^\top$ 
is taken into account. In resulting deformations are shown in ~\cref{fig:Ex123}. 
One observes that the microscopic pattern 
leads to very stiff behavior of the plate in the direction of the stripes and a very flexible behavior in the perpendicular direction.
We compare it with a corresponding paper experiment.
To this end, a sheet of paper  (green, $120g/m^2$) of size $28cm\times28cm$ is used with stripes of thick paper (red, $300 g/m^2$) of width $4mm$ where glued on top with a stripe distance of $4mm$. The resulting layered paper is clamped between two pieces of wood (cf. ~\cref{fig:Ex123}). 
\medskip

\noindent \emph{Microstructure reflecting radially arranged stripes.}  
Under the force and boundary conditions of the above experiment, we next consider a 
microstructure that mimics radial arranged rays of hard material centered at the midpoint 
$(0,0.5)$ of the left-hand side of the plate. The width of the stripes is proportional to the distance from the center.

First, we consider a stripe pattern similar to the one considered above with hard material on 
$(0,1) \times (\tfrac{1}{3},\tfrac{2}{3}) \times(-\tfrac{1}{2}, \tfrac{1}{2})$ and soft material elsewhere, where the ratio of the hard material $\tfrac13$ equals the ratio of the angle of a single hard material stripe and the angle 
of a single soft material stripe in the radially arranged strip pattern in the corresponding experiment. For a point $x \in \omega$ with 
$x-(0,0.5) = (r \cos \alpha, r \sin \alpha)$ for $\alpha \in (0,\pi)$ and $r\in \R_+$, we consider $\gamma = ar+b$, for $a,b\in \R$, such that $1 = \frac15a +b$ and $\frac{1}{10}= \frac{\sqrt{5}}{2}a+b$. This models the widening of the strips while maintaining the thickness of the plate.
Let us denote the resulting homogenized tensor by $\widehat{\mathcal{C}}^{2,\gamma}$ .
%and choose \MR{check this!} $\gamma=1$.
%\MR{rewrite description!} Hence, we considered a distribution of homogenized tensors as follows: 
%For $\alpha = \pi \frac{i}{31}$, and $\gamma = ar+b$, with $r = \frac{\sqrt{5}}{2}\frac{i}{31},i\in \{0,\dots,31\}$, 
%such that $1 = \frac15a +b$ and $\frac{1}{10}= \frac{\sqrt{5}}{2}a+b$, and a point $x \in [0,1]^2$ with 
%$|x - (0,\frac12)^\top| = r$, and $ x - (0,\frac12)^\top = r(\cos(\alpha - \frac\pi2),\sin(\alpha - \frac\pi2))^\top$.
%%%%%%%%%%

Then, applying the coordinate transformation formula  \cite[Section 3.4]{AlGePa19}
the resulting homogenized tensor $\mathcal{C}^{2,\gamma}(x)$ reflecting 
the radially arranged stripe pattern at $x \in \omega$ 
%with $x-(0,0.5) = (r \cos \alpha, r \sin \alpha)$ for $\alpha \in (0,\pi)$ and $r\in \R_+$
 is given by 
\begin{align*}
\mathcal{C}^{2,\gamma}(x) A : A \coloneqq&
\widehat{\mathcal{C}}^{2,\gamma} A_\alpha: A_\alpha \textit{ with }
 A_\alpha \coloneqq \begin{pmatrix}
 \cos\alpha & \sin\alpha \\
 -\sin\alpha & \cos\alpha
 \end{pmatrix}
 A
 \begin{pmatrix}
 \cos\alpha & -\sin\alpha \\
 \sin\alpha & \cos\alpha
 \end{pmatrix}
\end{align*}
for all $A\in \R^{2\times2}_\text{sym}$.

%The values of the homogenized quadratic form for different $x\in [0,1]^2$ were obtained by linear interpolation. 
%The same boundary conditions and force as in the previous experiments were applied. 
In ~\cref{fig:Ex123} in the top right, the resulting deformation is shown. 
Below,  on a $28cm\times28cm$ sheet of paper (purple, $120g/m^2$), 
we take into account 16  glued on stripes of thicker paper (yellow, $300 g/m^2$), radially symmetric, broadening with the distance from the  
clamped center part of the paper. 
The opening angle of the hard material (yellow) is $3.75^\circ$ and the opening angle of the soft material (purple)
is $7.5^\circ$ with the width of the hard material stripes ranges from $1mm$ to $1.2cm$.

%ffffffffffffff
\begin{figure}[htbp]
\centering
\includegraphics[width=\linewidth]{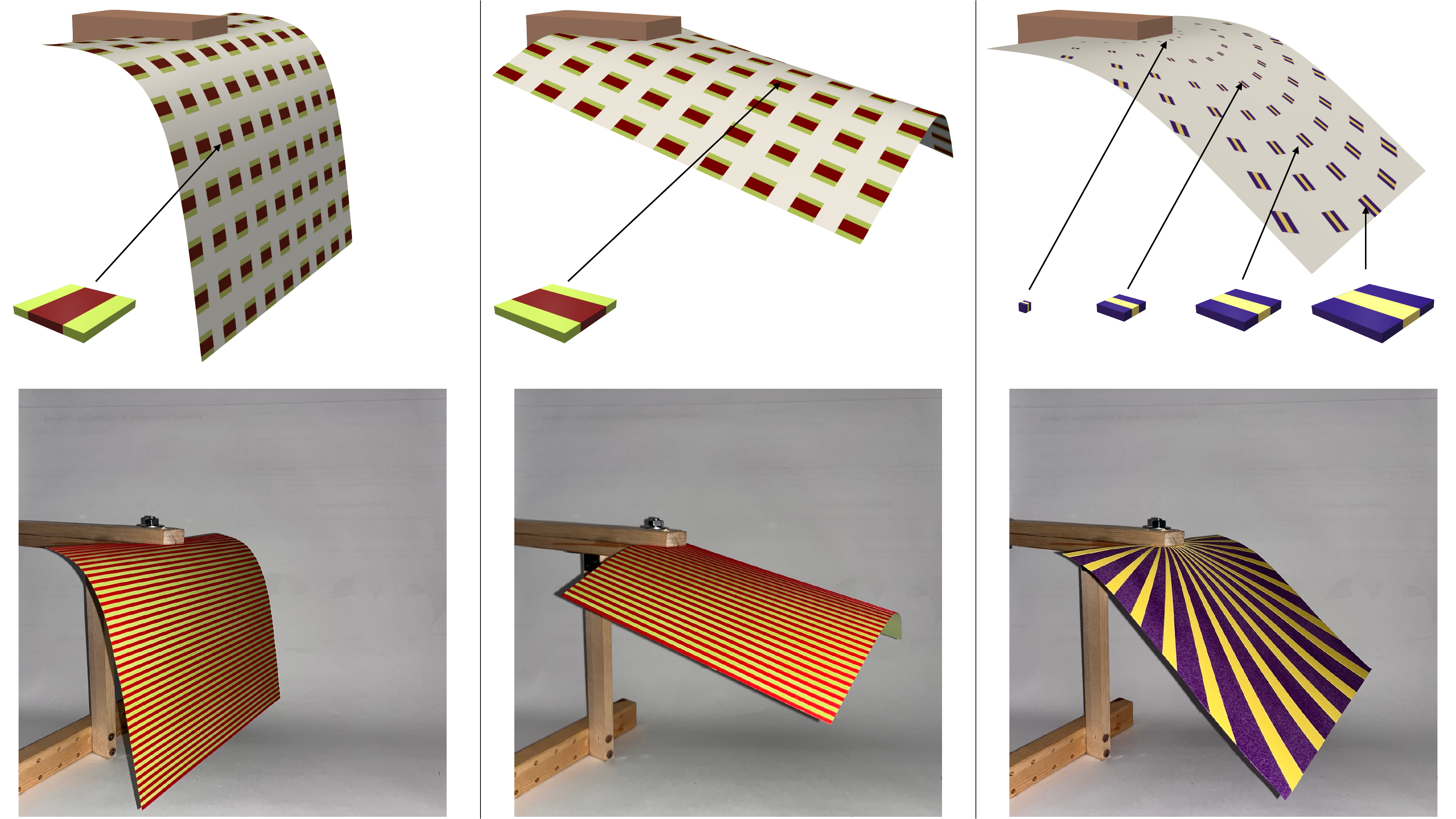}
\caption{Top row: deformed configurations under a uniform vertical load for stripe-type microstructures (left and middle) and a microstructure with radial rays (right). Bottom row: photos of physical experiments with stripes of thicker paper ($300 g/m^2$, left and middle: red, right: yellow) glued on thinner paper 
($120g/m^2$, left and middle: green, right: purple).
}
\label{fig:Ex123}
\end{figure}
%ffffffffffffff

\noindent \emph{Microstructure with periodic diagonal aligned stripe pattern.} 
Now, we consider a constantly distributed microstructure consisting of layers of hard and soft material oriented diagonally in the first two coordinates, i.e.the hard material phase is given by all $y = (y',y_3)$ with 
$$\left \vert y' - \begin{pmatrix} 1 \\ 0 \end{pmatrix} \right\vert_1 \in (\tfrac14,\tfrac34)\cap (\tfrac54,\tfrac74)\,$$
where $\vert\cdot\vert_1$ is the $1$-norm in $\R^2$.
The resulting homogenized tensor in Voigt-type notation is
\[ \mathcal{C}^{2,\gamma,v} = \begin{pmatrix}
0.17883  & 0.10699 & 0.11725 \\
0.10699 & 0.17883& 0.11725 \\
0.11725 & 0.11725 & 0.13436
\end{pmatrix}\,. \]
The deformation was enforced by the boundary conditions \eqref{eq:bdexpA} 
reflecting a compression in the $x_1$ direction by a factor $\tfrac{3}{8}$. No surface load applies. Again, a small random initial deformation is considered.
The computed deformation is shown in ~\cref{fig:Ex45} on the top left. The qualitative behavior is compared with a sheet of paper (blue, $120g/m^2$) of size $28cm\times28cm$, where stripes of thick paper (orange, $300 g/m^2$) of width $5mm$ were glued on, with a distance of $5mm$, in direction of one diagonal of the paper. A photo of this experiment is shown in ~\cref{fig:Ex45} on the left of the bottom row.
Due to the anisotropic microscopic material distribution the macroscopic elastic energy is also anisotropic leading both in the simulation and the experiment to a firm twist of the resulting valley formed by the bent plate.
\medskip

\noindent \emph{Microstructure with axes aligned and diagonal trusses of hard material.} 
In this experiment, we take into account a macroscopically varying microstructure with trusses of hard material on a soft material background.
In the rescaled model with vertically homogeneous material on $(0,1)^2 \times (-\tfrac12, \tfrac12)$  
the hard material phase is given by 
\begin{align*}
\big\{ y \;:\; & 
1- \tfrac{b}{\sqrt{2}} \vert y'- (1,0)^\top \vert_1  < 1+ \tfrac{b}{\sqrt{2}}  \vee
1- \tfrac{b}{\sqrt{2}} \vert y'\vert_1 < 1+ \tfrac{b}{\sqrt{2}} \vee \\
& y_1 < \tfrac{a}{2}  \vee  y_1 > 1- \tfrac{a}{2} \vee 
y_2 < \tfrac{a}{2}  \vee  y_2 > 1- \tfrac{a}{2}
\big\}
\end{align*}
with thicknesses $a(x_1)=(1-x_1) \frac{2-\sqrt{3}}{2}$, $b(x_1)=x_1 \frac{2-\sqrt{3}}{2}$ of the axes aligned and diagonal trusses, respectively.
The prototype of the fundamental cell and the resulting fine scale pattern with cell size $\varepsilon = \tfrac1{16}$ in $y'$ coordinates is depicted in 
~\cref{fig:Ex45} on the right. On the macroscale the boundary conditions \eqref{eq:bdexpA} and load are applied. 
Furthermore, for the paper experiment counterpart, the truss structure with $14^2$ $2cm\times2cm$ cells of thick paper ($300g/m^2$) is glued 
on a $28cm\times 28cm$ thin paper ($120g/m^2$). A comparison of the simulated deformation for the homogenized model and the experiment 
is shown as well in ~\cref{fig:Ex45}. It is clearly visible that the bent plate is no longer symmetric in $x_1$ direction and the bending is much stronger on the right.
%ffffffffffffff
\begin{figure}[htbp]
\centering
\begin{tikzpicture}
\node at (0,0){\includegraphics[width=0.98\linewidth]{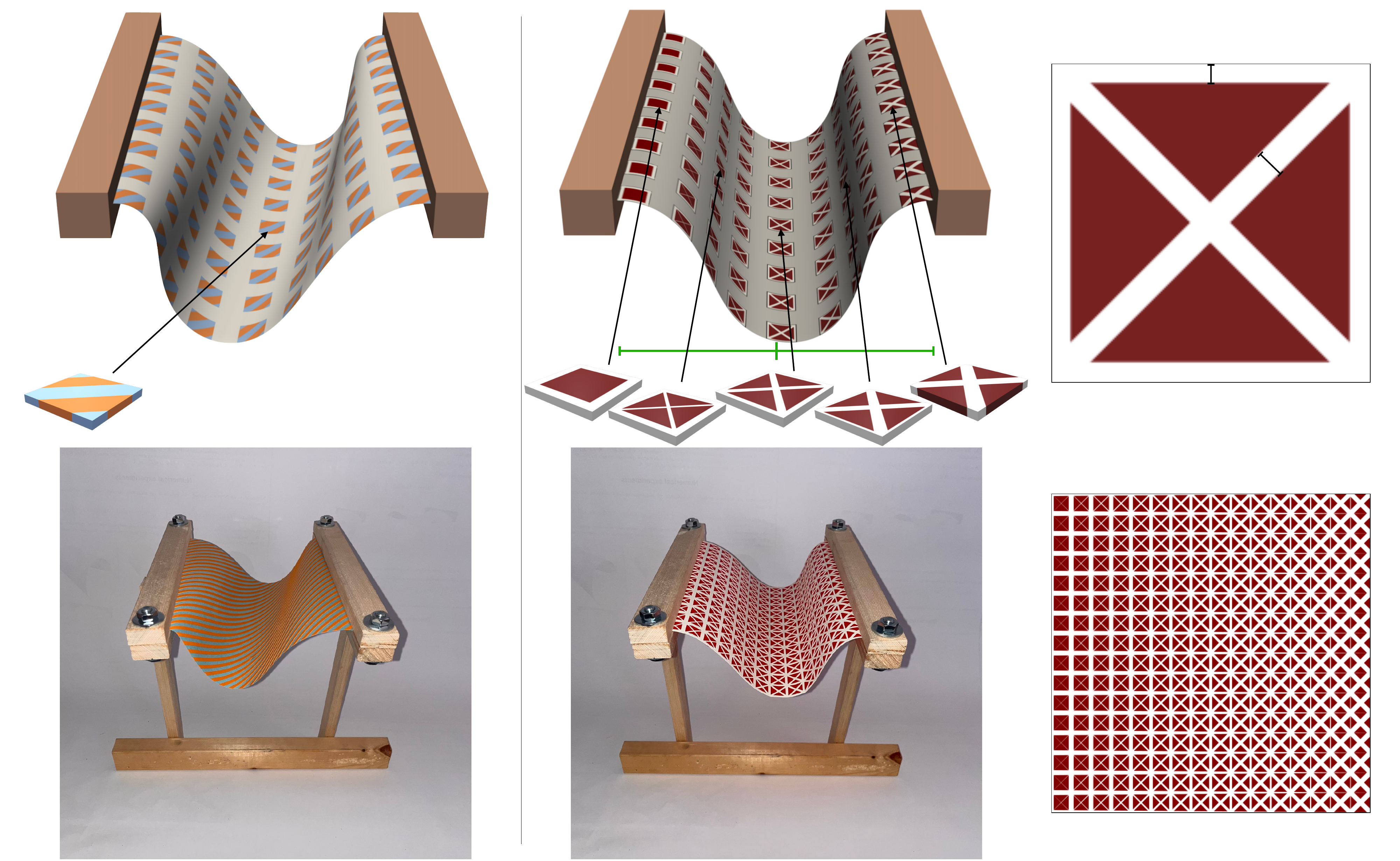}};
\node at (4.8,3.25) {$a$};
\node at (5.3,2.55) {$b$};
\end{tikzpicture}
\caption{Top left: deformed configuration under compression enforced by the boundary conditions \eqref{eq:bdexpA} with soft material in blue and hard material in orange, bottom left: corresponding experiment with stripes of thicker paper ($300g/m^2$, orange) glued on thin paper ($120g/m^2$blue); top middle: deformed configuration for a truss type microstructure under the same boundary conditions, with soft material in red and hard material in white, 
top right material distribution on the rescaled microscopic cell, bottom middle: a correspondingly deformed sheet of paper (red) with glued layer structure of the thicker paper (white), bottom right: undeformed experimental plate configuration viewed from above.}
\label{fig:Ex45}
\end{figure}
%ffffffffffffff

\bibliographystyle{siam}

\end{document}